\documentclass[a4paper,12pt,leqno]{amsart}
\usepackage{geometry}
\usepackage{mathtools, enumitem, etoolbox}
\usepackage[natbib=true, style=numeric, sorting=nyvt]{biblatex}
\usepackage{hyperref}
\numberwithin{equation}{section}
\newtheorem{theorem}{Theorem}
\newtheorem{corollary}{Corollary}
\newtheorem{proposition}{Proposition}[section]
\newtheorem{remark}[proposition]{Remark}
\newtheorem{lemma}[proposition]{Lemma}
\newcommand{\N}{\mathbb{N}}
\newcommand{\R}{\mathbb{R}}
\newcommand{\E}{\mathbb{E}}
\newcommand{\p}{\mathbb{P}}
\newcommand{\Z}{\mathbb{Z}}
\newcommand{\C}{\mathbb{C}}

\newcommand*\conj[1]{\overline{#1}}
\begin{document}
\begin{abstract}
For $f$ a Steinhaus random multiplicative function, we prove convergence in distribution of the appropriately normalised partial sums 
\[ 
\frac{{(\log \log x)}^{1/4}}{\sqrt{x}} \sum_{\substack{n \leq x \\ P(n) > \sqrt{x}}} f(n),
\]
where $P(n)$ denotes the largest prime factor of $n$. We find that the limiting distribution is given by the square root of an integral with respect to a critical Gaussian multiplicative chaos measure multiplied by an independent standard complex normal random variable.
\end{abstract}
\author{Seth Hardy}
\title[The distribution of random multiplicative functions]{The distribution of partial sums of random multiplicative functions with a large prime factor}
\address{Mathematics Institute, Zeeman Building, University of Warwick, Coventry CV4 7AL, England}
\email{Seth.Hardy@warwick.ac.uk}
\date{\today}
\thanks{The author is supported by the Swinnerton-Dyer scholarship at the Warwick Mathematics Institute Centre for Doctoral Training.}
\maketitle
\section{Introduction}
\subsection{Background}
With $\mu$ the Möbius function, the statement 
\[ 
\sum_{n \leq x} \mu(n) \ll_\varepsilon x^{1/2 + \varepsilon} ~~~ \text{for all } \varepsilon > 0
\]
is equivalent to the Riemann hypothesis. This fact motivated the 1944 paper of~\citet{Wintner}, where a probabilistic model for the Möbius function was investigated. Noting that the Möbius function is the unique multiplicative function supported on square-free integers that takes the value $-1$ on primes, Wintner took a collection of independent identically distributed random variables ${(f(p))}_{p \text{ prime}}$ such that $ \p (f(p) = \pm 1) = 1/2$, and considered the \emph{random multiplicative function} defined by
\[ 
f(n) = \prod_{p | n} f(p), 
\]
on square-free $n$, and $f(n) = 0$ otherwise. This function is known in the literature as the Rademacher random multiplicative function. In his paper, Wintner proved that for any $\varepsilon>0$ we almost surely have 
\[ 
\sum_{n \leq x} f(n) \ll_\varepsilon x^{1/2 + \varepsilon}, 
\]
proving a random analogue of the Riemann hypothesis. An alternative random model is the Steinhaus random multiplicative function: a model for Dirichlet characters and continuous characters. We define a Steinhaus random multiplicative function $f$ by taking ${(f(p))}_{p \text{ prime}}$ to be independent Steinhaus random variables; that is, random variables uniformly distributed on the unit circle $\{ z \in \mathbb{C} \colon |z| = 1 \}$, and setting 
\[ 
f(n) = \prod_{p \mid n} {f(p)}^{v_p(n)}, 
\]
where $v_p(n)$ denotes the $p$-adic valuation of $n$. \\
Random multiplicative functions have been studied extensively since Wintner. The main result of this paper provides a step towards the longstanding open problem concerning the distribution of the partial sums $\sum_{n \leq x} f(n)$, for $f$ a Steinhaus random multiplicative function:
\begin{theorem}\label{t:main}
Let $P(n)$ denote the largest prime factor of $n$, let $\gamma \coloneqq \lim_{n \rightarrow \infty} (\sum_{k=1}^n \frac{1}{k} - \log n )$ be the Euler--Mascheroni constant, and let $C = \frac{e^{- \gamma} \log 2}{2 \pi}$. For $f$ a Steinhaus random multiplicative function, we have
\[ 
\frac{{(\log \log x)}^{1/4}}{\sqrt{x}} \sum_{\substack{n \leq x \\ P(n) > \sqrt{x}}} f(n) \xrightarrow{d} \sqrt{C V_{\mathrm{crit}}} \, \mathcal{CN} (0, 1) \quad \text{as } x \rightarrow \infty , 
\]
where $\mathcal{CN} (0,1)$ is a standard complex Gaussian random variable and $V_{\mathrm{crit}}$ is a random variable that is independent of the Gaussian such that $\E [V_{\mathrm{crit}}^q]$ exists for each $0 < q < 1$. We define $V_{\mathrm{crit}}$ as a distributional limit of integrals with respect to a critical Gaussian multiplicative chaos measure. See Theorem~\ref{t:SWext} for details.
\end{theorem}
\begin{corollary}[Convergence of moments]\label{c:momentconvergence}
Let $f$ be a Steinhaus random multiplicative function. For any fixed $0 < q < 2$, there exists a constant $M = M(q) > 0$ such that
\[
\lim_{x \rightarrow \infty} \E {\biggl(\frac{{(\log \log x)}^{1/4}}{\sqrt{x}} \Bigl| \sum_{\substack{n \leq x \\ P(n) > \sqrt{x}}} f(n) \Bigr| \biggr)}^{q} = M .
\]
Specifically, we find that $M = C^{\frac{q}{2}} \E \bigl[ V_{\mathrm{crit}}^{\frac{q}{2}} \bigr] \Gamma (\frac{q}{2} + 1)$, where $C$ and $V_{\mathrm{crit}}$ are as in Theorem~\ref{t:main}, and $\Gamma$ is the gamma function.
\end{corollary}
\begin{corollary}[Heavy tails]\label{c:heavytails}
For $f$ a Steinhaus random multiplicative function, and for any $\varepsilon > 0$, $y \geq 2$, we have 
\[ 
\frac{1}{y^{2+\varepsilon}} \ll_{\varepsilon} \p \biggl( \frac{{(\log \log x)}^{1/4}}{\sqrt{x}} \Bigl| \sum_{\substack{n \leq x \\ P(n) > \sqrt{x}}} f(n) \Bigr| > y \biggr) \ll_{\varepsilon} \frac{1}{y^{2-\varepsilon}} ,
\]
when $x$ is sufficiently large in terms of $y$.
\end{corollary}
This second corollary tells us that the tails of our partial sums are dominated by the tails of $\sqrt{V_{\mathrm{crit}}}$. Similar tail estimates for the full sum have been proven by~\citet[Corollary~2]{HarperLM}, but the above gives a more precise lower bound when $y$ is bounded.
\begin{remark} 
Our results handle a positive proportion of the full sum, seeing as the set $\{ 1 \leq n \leq x \colon P(n) > x^{1/2} \}$ has positive asymptotic density.
\end{remark}
\begin{remark}
We do not prove our main theorem in the Rademacher case owing to the covariance structure of $\Re \sum_{p \leq x} \frac{f(p)}{p^{1/2 + it}}$, which we discuss more in Section~\ref{s:SW Outline}. This obstruction can likely be overcome but lies outside the scope of this paper. We do, however, prove one of our main results (Theorem~\ref{t:variance}) in the Rademacher case, and this result may be of independent interest.
\end{remark}
\subsection{Previous work}
Since the inception of random multiplicative functions in~\cite{Wintner}, much research has been done to try and understand the almost sure size of the partial sums $\sum_{n \leq x} f(n)$, for $f$ a Rademacher or Steinhaus random multiplicative function. The upper bound was investigated in~\cite{Basquin, LTW, Mastrostefano, Wintner}, culminating in the recent work of~\citet{Caich}, who showed that, for any $\varepsilon > 0$, we almost surely have
\[ 
\Bigl| \sum_{n \leq x} f(n) \Bigr| \ll \sqrt{x} {(\log \log x)}^{3/4 + \varepsilon}, 
\]
for $f$ a Rademacher or Steinhaus random multiplicative function. For lower bounds~\citet{HarperLargeFluct} (improving on earlier work~\cite{HarperOmega}) has shown that, for any function $V(x)$ that diverges to infinity,
\[ 
\Bigl| \sum_{n \leq x} f(n) \Bigr| \geq \sqrt{x} \frac{{(\log \log x)}^{1/4}}{V(x)}, 
\]
for arbitrarily large $x$, almost surely, for $f$ a Rademacher or Steinhaus random multiplicative function. Note that~\citet{Mastrostefano} proved a matching upper bound for $\sum_{n \leq x, \, P(n) > \sqrt{x}} f(n)$, so this result is expected to reflect the true size of the fluctuations. \\
Another central problem has been to understand the expected size of the partial sums, $\E \bigl| \sum_{n \leq x} f(n) \bigr|$, for $f$ a Rademacher or Steinhaus random multiplicative function. It was conjectured by~\citet{Helson} that this quantity should enjoy ``better than square-root'' cancellation, in that $\E \bigl| \sum_{n \leq x} f(n) \bigr| = o(\sqrt{x})$. This conjecture was motivated by the relationship
\[ 
\lim_{T \rightarrow \infty} \frac{1}{T} \int_{0}^{T} \Bigl| \sum_{n \leq x} n^{it} \Bigr|^{2q} \, dt = \E \Bigl| \sum_{n \leq x} f(n) \Bigr|^{2q} ,
\]
that holds for $q > 0$. Seeing as the sums $\sum_{n \leq x} n^{it}$ are multiplicative analogues of Dirichlet kernels, which possess small $L^1$ norms, Helson conjectured that the $L^1$ norm of random multiplicative functions should also be small. The conjecture was the subject of much investigation~\cite{BondSeipExp, HarNikRadExp, WeberExp}
and, in 2020,~\citet{HarperLM} provided a positive answer to Helson's conjecture, showing that 
\begin{equation}\label{equ:btsrc}
\E \Bigl| \sum_{n \leq x} f(n) \Bigr| \asymp \frac{\sqrt{x}}{{(\log \log x)}^{1/4}}, 
\end{equation}
using an intricate probabilistic argument relating to the ballot problem. One of the key things noticed in this work was that the partial sums of random multiplicative functions are related to the theory of \emph{Gaussian multiplicative chaos}, an area of probability theory introduced by~\citet{Kahane} in the 1980s. \\
One of the biggest remaining open problems is to understand the \emph{distribution} of the partial sums $\sum_{n \leq x} f(n)$. It was first shown by~\citet{HarperDist} that the partial sums \emph{do not} satisfy a central limit theorem, in that $\frac{1}{\sqrt{x}} \sum_{n \leq x} f(n)$ does not converge in distribution to a (non-degenerate) Gaussian random variable, and, improving on previous work of~\citet{HoughLimitingDist}, Harper also showed that we do have a central limit theorem if we restrict to integers with few prime factors. Various other central limit theorems have been proven when one restricts to a special set of integers~\cite{ChatSoundCLT, SoundXu}, and also for exponential sums~\cite{BNR}. \\ 
In a recent pre-print,~\citet[Theorem~1.1]{GW} use martingale techniques to establish the first non-Gaussian convergence theorem for (appropriately weighted) partial sums of random multiplicative functions. Their result is quite general and applies to sums of the form $ {( {\E |\sum_{n \leq x} g(n) f(n)|^2} )}^{-1/2} \sum_{n \leq x} f(n) g(n)$ where $f$ is a Rademacher or Steinhaus random multiplicative function and $g$ is a multiplicative function with sparse support or good decay. These conditions exclude, for example, $g \equiv 1$, where one expects slightly different behaviour more in line with this paper (see the end of Section~\ref{s:GMC} for more on this). Indeed, our result is novel in the sense that our normalised partial sums \emph{are not} of the form $ {( {\E |\sum_{n \leq x} a_n f(n)|^2} )}^{-1/2} \sum_{n \leq x} a_n f(n)$ for some complex coefficients $a_n$, as in all previous works. Such sums are sensible to study in most cases, but are not the correct thing to study in the case of the full partial sums $\sum_{n \leq x} f(n)$, or some close analogue, as in our case. For example, a corollary of~\eqref{equ:btsrc} is that (in the Steinhaus case, say) $ {( {\E |\sum_{n \leq x} f(n)|^2} )}^{-1/2} \sum_{n \leq x} f(n) = \frac{1}{\lfloor \sqrt{x} \rfloor} \sum_{n \leq x} f(n) \xrightarrow{p} 0$, and to investigate the distribution of the partial sums one should instead consider $\frac{{(\log \log x)}^{1/4}}{\sqrt{x}} \sum_{n \leq x} f(n) $, or $ \frac{{(\log \log x)}^{1/4}}{\sqrt{x}} \sum_{n \leq x, \, P(n) > \sqrt{x}} f(n) $ in our case\footnote{If one normalises by anything smaller than $\frac{{(\log \log x)}^{1/4}}{\sqrt{x}}$, an application of Markov's inequality gives convergence in probability to zero. Furthermore, an adapted Payley--Zygmund inequality allows one to show $\p (\frac{{(\log \log x)}^{1/4}}{\sqrt{x}} \bigl| \sum_{n \leq x} f(n) \bigr| > c) \gg_c 1 $ for $c$ in a bounded range.}. This motivated~\citet[Conjecture~1.6]{GW} to conjecture that we should have
\[
\frac{{(\log \log x)}^{1/4}}{\sqrt{x}} \sum_{n \leq x} f(n) \xrightarrow{d} \sqrt{\tilde{V}_{\mathrm{crit}}} \, \mathcal{CN} (0,1) \quad \text{as } x \rightarrow \infty,
\]
where $\tilde{V}_{\mathrm{crit}}$ is a limiting object defined using partial Euler products. Theorem~\ref{t:main} provides evidence towards this conjecture and explicitly relates (our analogue of) $\tilde{V}_{\mathrm{crit}}$ to Gaussian multiplicative chaos. We remark that the recent work of~\citet{Caich} builds on the previous work of~\citet{Mastrostefano} on almost sure upper bounds for $\sum_{n \leq x, \, P(n) > \sqrt{x}} f(n) $, so handling this subsum has proved a useful step towards understanding the full sum.

\subsection{Key results}\label{s:intromc}
For $f$ a Steinhaus random multiplicative function, we begin by noting that
\[ 
\frac{{(\log \log x)}^{1/4}}{\sqrt{x}} \sum_{\substack{n \leq x \\ P(n) > \sqrt{x}}} f(n) = \frac{{(\log \log x)}^{1/4}}{\sqrt{x}} \sum_{\sqrt{x} < p \leq x} f(p) \sum_{m \leq x/p} f(m) .
\]
Conditioning on the innermost sum, the right-hand side is a sum of independent random variables whose distribution approximates a complex Gaussian with mean $0$ and variance $\frac{{(\log \log x)}^{1/2}}{x} \sum_{\sqrt{x} < p \leq x } \bigl| \sum_{m \leq x/p} f(m) \bigr|^2 $. This Gaussian approximation is the stage where restricting to integers $n$ with $P(n) > \sqrt{x}$ is important. One may suspect, by Parseval's theorem or Perron's formula, that the variance is roughly
\begin{equation}\label{equ:variances}
\frac{{(\log \log x)}^{1/2}}{\log x} \int_{\R} \biggl| \frac{F (1/2 + it)}{1/2+it} \biggr|^2 \, dt ,
\end{equation}
where $F (s) \coloneqq \prod_{p \leq \sqrt{x}} {\bigl( 1 - \frac{f(p)}{p^{s}} \bigr)}^{-1}$ is the partial Euler product associated to $f$. An analogous statement should hold in the Rademacher case, and in this case, we define $F (s) \coloneqq \prod_{p \leq \sqrt{x}} {\bigl( 1 + \frac{f(p)}{p^{s}} \bigr)}$ instead. Making this statement precise is a key step toward proving Theorem~\ref{t:main}, and constitutes the bulk of the paper. It turns out to be easier to prove a similar statement involving significantly shorter Euler products, showing that the distribution is controlled by the behaviour of $f(p)$ on relatively small primes. To state the theorem, we define the truncated Euler products $F_j (s) \coloneqq \prod_{p \leq \sqrt{x}^{e^{-j}}} {\bigl( 1 - \frac{f(p)}{p^s} \bigr)}^{-1}$ in the Steinhaus case and $F_j (s) \coloneqq \prod_{p \leq \sqrt{x}^{e^{-j}}} {\bigl( 1 + \frac{f(p)}{p^s} \bigr)}$ in the Rademacher case, so that $F_0 = F$. We prove the following:
\begin{theorem}\label{t:variance}
Let $f$ be a Steinhaus or a Rademacher random multiplicative function, let $C = \frac{e^{-\gamma} \log 2}{2\pi}$,  $y = \lfloor 100 \log \log \log x \rfloor$, and $B(x) = \sqrt{x}^{e^{-y}}$, so that $F_y (s)$ is a product over primes below $B(x)$. Then with probability $1 - O ( {(\log \log x)}^{-1/5} )$, we have
\[
\frac{{(\log \log x)}^{1/2}}{x} \sum_{\sqrt{x} < p \leq x} \bigl| \sum_{m \leq x/p} f(m) \bigr|^2 = \frac{C {(\log \log B(x))}^{1/2}}{\log B(x)} \int_{\R} \biggl| \frac{F_{y} (1/2 + it)}{1/2 + it} \biggr|^2 \, dt + E(x) ,
\]
where $E(x)$ satisfies $E(x) \ll {(\log \log x)}^{-1/5}$.
\end{theorem}
Previously it was only known (using arguments from~\cite{HarperLM}) that the low moments of these quantities should be of the same order of magnitude. Now, having shown that our variances concentrate around the mean square of an Euler product, we are in a hopeful position, since previous work of~\citet[Theorem~1.9]{SaksmanWebb} tells us that, when $t$ is restricted to a bounded interval, we have convergence in distribution to an integral with respect to a random measure known as a critical Gaussian multiplicative chaos measure. In the following theorem, we extend their result to give convergence over the full range of integration. To state our result, we redefine the Euler product over primes $p \leq x$ by $F_{(x)} (s) \coloneqq \prod_{p \leq x} {\bigl( 1 - \frac{f(p)}{p^{s}} \bigr)}^{-1}$ in the Steinhaus case. This is a notational convenience that makes the statement of our theorem more natural.
\begin{theorem}[Extension of Theorem~1.9 of~\citet{SaksmanWebb}]\label{t:SWext}
For $g \in C^l [a,b]$, denote $\| g \|_{C^l [a,b]} = \sum_{j=0}^l \| g^{(j)} \|_{L^{\infty} [a,b]}$. For $f$ a Steinhaus random multiplicative function, we have
\[ 
\frac{{(\log \log x )}^{1/2}}{\log x} \int_{\R} \biggl| \frac{F_{(x)} (1/2 + it)}{1/2 + it} \biggr|^2 \, dt \xrightarrow{d} V_{\mathrm{crit}} \quad \text{as } x \rightarrow \infty,
\]
where $\E [V_{\mathrm{crit}}^q]$ exists for each $0 < q < 1$. $V_{\mathrm{crit}}$ is defined as the limit (in distribution) of $V_n = \int_{-n}^n \frac{g(t)}{|1/2 + it|^2} \lambda (dt)$ as $n \rightarrow \infty$, where $g (t)$ is a positive random continuous function such that, for each $n > 0$, all norms $\| g \|_{C^l [-n,n]}$ and $\| 1/g \|_{C^l [-n,n]}$ possess moments of all orders, and $\lambda (dt)$ is a critical Gaussian multiplicative chaos measure. Note that $g$ and $\lambda$ may not be independent. The random variable $V_{\mathrm{crit}}$ is precisely the one that appears in Theorem~\ref{t:main}.
\end{theorem}
For details of this weak convergence, see Section~\ref{s:distribofvar}, and more specifically Section~\ref{s:completionoffullintconv}. An analogous theorem is not known in the Rademacher case (and would involve non-trivial adjustments to the argument). For this reason, we do not give a proof of Theorem~\ref{t:main} in the Rademacher case. 

\subsection{Connection to Gaussian multiplicative chaos}\label{s:GMC}
We now explain the connection between our variance (which we think of as mean squares of random Euler products like~\eqref{equ:variances}) and the theory of Gaussian multiplicative chaos, restricting ourselves to the Steinhaus case. For an introduction to the theory, see~\citet{RhodesVargas}. A similar discussion can be found in the introduction of~\citet{HarperLM}. Here, we put a particular emphasis on understanding the limiting distribution of our variances. \\
We begin by noting that, for fixed $t$, the quantity $\log |F (1/2 + it)| \approx \Re \sum_{p \leq \sqrt{x}} \frac{f(p)}{p^{1/2+it}} $ is distributed like a Gaussian with mean $0$ and variance approximately $ \frac{1}{2} \sum_{p \leq \sqrt{x}} \frac{1}{p} = \frac{1}{2} \log \log x + O(1)$. Consequently, we think of $X(t) = \log | F(1/2 + it)|$ as a Gaussian field. The reader is reminded that a key property of Gaussian fields is that their distribution is completely determined by their means, $\E [X(t)]$ (which will be zero in the Steinhaus case), and their covariances, $\E [ X(t_1) X(t_2) ]$. With this in mind, we rewrite our variances in~\eqref{equ:variances} as $\frac{{(\log \log x)}^{1/2}}{\log x} \int_{0}^{1} e^{2 X(t)} \, dt $. For simplicity, we have restricted ourselves to a compact interval and ignored the denominator. It is natural to normalise our integral so that the integrand has expectation one, and, if $X(t)$ were actually Gaussian for each $t \in [0,1]$, the correct\footnote{This follows from the moment generating function for Gaussian random variables.} renormalisation would be to multiply the integrand by $e^{- 2 \E {X(t)}^2}$. This accounts for the factor of $\frac{1}{\log x}$ in~\eqref{equ:btsrc}, so our variances are roughly
\begin{equation}\label{equ:varianceswithexp}
{(\log \log x)}^{1/2} \int_{0}^{1} \exp \Bigl( 2 X(t) - 2 \E \bigl( {X(t)}^2 \bigr) \Bigr) \, dt ,
\end{equation}
where $X(t) = \log |F(1/2 + it)|$. Broadly speaking, Gaussian multiplicative chaos is the study of random measures
\begin{equation*}
\mu_\gamma (dt) = \exp \Bigl( \gamma \tilde{X} (t) - \frac{\gamma^2}{2} \E \bigl({\tilde{X} (t)}^2\bigr) \Bigr) \, dt ,
\end{equation*}
where $\gamma \in \R$, $\mathcal{D} \subseteq \R$ is some domain, and ${\bigl( \tilde{X} (t) \bigr)}_{t \in \mathcal{D}}$ is Gaussian field with a certain type of \emph{logarithmic} covariance structure. Note that these measures are similarly normalised so that $\E \mu_\gamma (dt) = dt$. It happens that $X(t) = \log |F (1/2 + it)|$ does possess an approximately logarithmic covariance structure, in that
\begin{align*}
\E \bigl[ X(t_1) X(t_2) \bigr] &\approx \sum_{p \leq \sqrt{x}} \frac{\cos\bigl( (t_1 - t_2) \log p \bigr)}{2p},
\end{align*}
which approaches $-\frac{1}{2}\log |t_1 - t_2|$ when $x$ is large and when $\frac{1}{\log x} \leq |t_1 - t_2| \leq 1$, say. This is (up to scaling) the covariance structure that is central in the study of Gaussian multiplicative chaos, so the theory can be used to analyse~\eqref{equ:varianceswithexp}. From a number theoretic standpoint, the logarithmic covariance structure is a manifestation of the fact that $p^{it} = e^{it \log p}$ rotate at different speeds for different primes $p$. Let $\mathcal{K} = \lfloor \log \log x \rfloor$, and write
\[ 
\log |F (1/2 + it)| \approx \Re \Biggl( \sum_{1 \leq k \leq \mathcal{K}} \sum_{\sqrt{x}^{e^{-(k+1)}} < p \leq \sqrt{x}^{e^{-k}}} \frac{f(p)}{p^{1/2 + it}} \Biggr) . 
\]
For the largest $k$, the innermost sum is roughly fixed as $t$ varies over bounded intervals. However, for the smallest $k$, the innermost sum is roughly fixed over intervals of length $\approx \frac{1}{\log x}$. Thus, there are $\log \log x$ ``scales'', and on the scale $k$, the innermost sum ``takes'' $\frac{2^{k}}{\log x}$ values as $t$ varies over $[-1,1]$, say. An important idea underlying the proof of Theorem~\ref{t:variance} is that the values of $f(p)$ on small primes dictate the bulk of the distribution of our sum, seeing as the average over $t$ has little effect on them. For more about $\log |F (1/2 + it)|$ and logarithmically correlated fields, see the survey article of~\citet{BaileyKeatingLogCorr}. \\
Theorem~\ref{t:SWext} makes explicit the relation between our variances and Gaussian multiplicative chaos and specifically involves a random measure known as the~\emph{critical} Gaussian multiplicative chaos measure. We do not explicitly define this object in this paper, but to gain some insight into its behaviour, we consider the more general expression~\eqref{equ:varianceswithexp} and describe how the parameter $\gamma$ influences the behaviour of
\begin{equation}\label{equ:critparam}
\int_{0}^1 \exp \Bigl( \gamma X(t) - \frac{\gamma^2}{2} \E \bigl( {X(t)}^2 \bigr) \Bigr) \, dt ,
\end{equation}
where $X(t) = \log |F(1/2 + it)|$. When $\gamma = 2$ this is precisely~\eqref{equ:varianceswithexp} without the factor ${(\log \log x)}^{1/2}$. Assuming that $X(t)$ is distributed like a Gaussian with mean $0$ and variance $\frac{1}{2} \log \log x$, say, one can find that the dominant contribution to the expected value of~\eqref{equ:critparam} comes from points $t$ where $X(t) = \frac{\gamma}{2} \log \log x + O(\sqrt{\log \log x}) $. Note that~\citet[Theorem~1.2]{ABH} have shown that
\[ 
\max_{t \in [0,1]} \log | F (1/2 + it)| = \log \log x - \frac{3}{4} \log \log \log x + o(\log \log \log x), 
\]
with probability $1-o(1)$, as $x \rightarrow \infty$. Subsequently, when $\gamma = 2$, the dominant contribution comes from the most extreme values, and the analysis of this case inherits the name critical Gaussian multiplicative chaos. Similar ideas suggest that for $\gamma > \gamma_c = 2$ we have convergence in probability to zero, since then we do not expect to find any points in $[0,1]$ for which $X(t) > \frac{\gamma}{2} \log \log x$. Thus, in our case, one can think of a critical Gaussian multiplicative chaos measure as the limiting measure supported on $t$ where $\log |F(1/2+it)|$ takes its largest values. However, as we have mentioned, to obtain a non-trivial limiting distribution when $\gamma = 2$, one needs to multiply~\eqref{equ:critparam} by ${(\log \log x)}^{1/2}$. This factor, which we initially motivated using the work of~\citet{HarperLM}, can also be motivated by the \emph{derivative martingale} (see~\citet{DupBerRemSheVar}, for example). Since it can be shown that
\[
M(\gamma) \coloneqq \int_{0}^{1} \exp \Bigl( \gamma \log |F (1/2+it)| - \frac{\gamma^2}{2} \E \bigl( \log |F(1/2+it)|^2 \bigr) \Bigr) \, dt
\]
almost surely converges to a non-zero quantity for $\gamma < 2$, and almost surely converges to zero for $\gamma \geq 2$, we have a (negative) jump in the limiting object at $\gamma = 2$. Thus, one may expect that the negative of the derivative at $\gamma = 2$ may be non-trivial in the limit, leading one to the study of
\[ 
-M'(2) = \int_{0}^{1} \bigl( 2 \E \bigl( {X(t)}^2\bigr) - X(t) \bigr) \exp \Bigl( 2 X(t) - 2 \E \bigl( {X(t)}^2 \bigr) \Bigr) \, dt,
\]
where $X(t) = \log |F(1/2 + it)|$. We previously saw that, without this pre-factor, the dominant contribution to the expected size of $M(2)$ comes from points $t \in [0,1]$ where $ X(t) = \log \log x + O ({(\log \log x)}^{1/2})$. Assuming these points still give the dominant contribution, and seeing as $2 \E [ {X(t)}^2 ] = \log \log x + O(1)$, the factor we have introduced increases the probable size of $M(2)$ by a factor of ${(\log \log x)}^{1/2}$, suggesting that this is the correct renormalisation. For a survey of critical chaos, see~\citet{PowellCGMC}. \\ 
We finish this section by remarking that the recent work of~\citet{GW} handles certain subcritical $(\gamma < 2)$ regimes, and so they do not need to renormalise their partial sums non-trivially to obtain a non-trivial distribution. See also~\citet{ArgHarKisHighPoints} for subcritical results about random Euler products.

\section{Outline of the proof of Theorem~\ref{t:variance} and structure of the paper}\label{s:outline}
In this section, we describe the organisation of the paper and sketch the proof of Theorem~\ref{t:variance}. Sections~\ref{s:nttools} and~\ref{s:probtools} will collect useful number theoretic and probabilistic tools, respectively. In Section~\ref{s:normalapprox}, we perform our conditional Gaussian approximation in the Steinhaus case, which follows from a result of~\citet{HarLam}. As mentioned, this delivers a complex Gaussian with zero mean and variance
\[ 
\frac{{(\log \log x)}^{1/2}}{x} \sum_{\sqrt{x} < p \leq x } \bigl| \sum_{m \leq x/p} f(m) \bigr|^2 .
\]
In Section~\ref{s:varianceanalysis}, we prove Theorem~\ref{t:variance}, which involves performing numerous approximations to this variance. We outline the key ideas here, restricting ourselves to the Steinhaus case (though the theorem is also proved in the Rademacher case). We note that a similar analysis has previously been conducted by~\citet[Section~3.2]{HarperLargeFluct}. First of all, an application of Perron's formula allows us to find that our variance is roughly
\[
\frac{{(\log \log x)}^{1/2}}{4 \pi^2} \int_{-x^{1/3}}^{x^{1/3}} \int_{- x^{1/3}}^{x^{1/3}} \frac{F (1/2 + i t_1) \conj{F (1/2 + i t_2)}}{(1/2 + i t_1)(1/2 - i t_2)} \sum_{\sqrt{x} < p \leq x} \frac{x^{i (t_1 - t_2)}}{p^{1+i(t_1-t_2)}} \, d t_1 \, d t_2 .
\]
The contribution to the above integral from points where $t_1 \not \approx t_2$ is small with high probability due to the decorrelation of the Euler products and the decay from the prime number sum. This allows us to reduce the above integral to points where $t_1 = t_2 + O \bigl( \frac{{(\log \log x)}^5}{\log x} \bigr)$. Over this reduced range, the contribution from the small primes to the Euler products $F (1/2 + i t_1)$ and $F (1/2 + i t_2)$ is roughly the same. Specifically, with $y = \lfloor 100 \log \log \log x \rfloor$, we have $\frac{F_y (1/2 + it_1)}{1/2 + it_1} \approx \frac{F_y (1/2 + it_2)}{1/2 + it_2}$. Therefore, if we define $ I_y (t) \coloneqq \prod_{\sqrt{x}^{e^{-y}} < p \leq \sqrt{x}} {\bigl( 1 - \frac{f(p)}{p^{1/2 + it}} \bigr)}^{-1}$ in the Steinhaus case, we find that, with high probability, our variance is roughly
\[ 
\frac{{(\log \log x)}^{1/2}}{4 \pi^2} \iint_{\substack{|t_1 - t_2| \leq \frac{{(\log \log x)}^{5}}{\log x} \\ |t_1|, |t_2| \leq {(\log \log x)}^4}} \frac{|F_{y} (1/2 + i t_1)|^2 I_y (t_1) \conj{I_y (t_2)} }{|(1/2 + it_1)|^2} \sum_{\sqrt{x} < p \leq x} \frac{x^{i(t_1 - t_2)}}{p^{1 + i (t_1 - t_2)}} \, d t_1 \, d t_2 ,
\]
where we have reduced the range of $t_1$ and $t_2$ using the decay in the denominator. Getting to this stage is the contents of Propositions~\ref{p:perron} and~\ref{p:smallprimes}. \\ 
Having obtained the integral of a mean square of an Euler product, we are closer to being able to apply Theorem~\ref{t:SWext} to deduce the distribution. However, our expression still involves the tails of the Euler products, $I_y (t_1)$ and $I_y (t_2)$, which we wish to remove. The key observation that allows us to do this is to notice that, as $t_1$ and $t_2$ vary, the larger primes rotate quickly. This suggests that we should typically have $I_y (t_1) \conj{I_y (t_2)} \approx \E I_y (t_1) \conj{I_y (t_2)}$ over our domain of integration, and an application of Chebyshev's inequality allows us to show that this is the case. This is the content of Proposition~\ref{p:smallprimesconc}, and, for it to succeed, we make use of barrier events from~\cite{HarperLargeFluct} (see~\cite[Section~3]{HarperMomentsIII} for an explanation of these). In this step, it is essential for $y$ to be relatively small (certainly $o (\log \log x)$), seeing as the behaviour of $f(p)$ on small primes \emph{does} influence the distribution of our variance. After replacing $I_y (t_1) \conj{I_y (t_2)}$ by $\E I_y (t_1) \conj{I_y (t_2)}$, we integrate over $t_2$ and find that our variance concentrates around
\[ 
\frac{e^{-\gamma} \log 2}{2 \pi} \frac{{({\log \log B(x)})}^{1/2}}{\log B(x)} \int_{\R} \biggl| \frac{F_{y} (1/2 + i t_1)}{1/2 + it_1} \biggr|^2 d t_1 , 
\] 
with high probability, where $B(x) \coloneqq \sqrt{x}^{e^{-y}}$. This is the content of Proposition~\ref{p:roughestapplication}. Combining the propositions in Section~\ref{s:varianceanalysis} delivers Theorem~\ref{t:variance}. \\
In Section~\ref{s:distribofvar}, we prove Theorem~\ref{t:SWext}. As mentioned, this is an extension of the work of~\citet[Theorem~1.9]{SaksmanWebb} to the weighted integral over the real line. We outline the main ideas of their argument at the beginning of Section~\ref{s:distribofvar}. To extend their result, we employ a tightness argument for integrals over compact domains followed by an approximation lemma that gives convergence in distribution for the full integral over the real line. \\ 
In Section~\ref{s:completion}, we use Proposition~\ref{p:steincna} and Theorems~\ref{t:variance} and~\ref{t:SWext} to prove Theorem~\ref{t:main}. We also prove Corollaries~\ref{c:momentconvergence} and~\ref{c:heavytails}.

\section{Preliminary number theory results}\label{s:nttools}
In this section, we collect some standard number theoretic results that will be used throughout the paper. The first two estimates give mean and pointwise bounds for sums involving prime numbers, respectively:
\begin{lemma}[Mean value over large primes]\label{l:meansqlargeprimes}
 Let $\Lambda (n)$ denote the von Mangoldt function. Uniformly for any sequence of complex numbers ${(a_n)}_{n=1}^{\infty}$ and any $T \geq 1$, we have
\[ 
\int_{-T}^{T} \biggl| \sum_{T^{1.01} \leq n \leq x} \frac{a_n \Lambda (n)}{n^{1+it}} \biggr|^2 \, dt \ll \sum_{T^{1.01} \leq n \leq x} \frac{| a_n |^2 \Lambda (n)}{n}.
\] 
\end{lemma}
\begin{proof}
This is Number Theory Result 1 of~\citet{HarperLargeFluct}.
\end{proof} 
\begin{lemma}[Pointwise bound for sums over primes]\label{l:pwprimebound}
For any $100 \leq x \leq y$, $t \neq 0$, we have
\[ 
\biggl| \sum_{x < p \leq y} \frac{1}{p^{1+it}} \biggr| \leq \frac{3}{|t| \log x} + O \Bigl( \bigl( 1 + |t| \bigr)e^{-c \sqrt{\log x}} \Bigr), 
\]
for some positive constant $c > 0$.
\end{lemma}
\begin{proof}
This follows from Number Theory Result 2 of~\citet{HarperLargeFluct}.
\end{proof}
To prove Proposition~\ref{p:roughestapplication}, we will utilise the following lemma:
\begin{lemma}[Perron integral for rough numbers]\label{l:roughintest}
Suppose that $B(x)$ is a function satisfying $\frac{\log x}{{(\log \log x)}^{1000}} \ll \log B(x) \ll \frac{\log x}{{(\log \log x)}^{2}}$, say. Let $W(x)$ be a function satisfying ${(\log \log x)}^2 \leq W(x) \leq {(\log x)}^{1/3}$ when $x$ is large. Then
\[ 
\int_{|u| \leq \frac{W(x)}{\log x}} \prod_{B(x) < q \leq \sqrt{x}} {\biggl(1 - \frac{1}{q^{1+iu}} \biggr)}^{-1} \sum_{\sqrt{x} < p \leq x} \frac{x^{iu}}{p^{1+iu}} \, du = \frac{2 \pi e^{-\gamma} \log 2}{\log B(x)} \biggl( 1 + O \biggl( \frac{1}{\log \log x} \biggr) \biggr) ,
\]
where the product over $q$ runs over primes, and $\gamma$ is the Euler--Mascheroni constant. Furthermore, an identical estimate holds if 
\[ 
\prod_{B(x) < q \leq \sqrt{x}} {\biggl(1 - \frac{1}{q^{1+iu}} \biggr)}^{-1} 
\]
is replaced by 
\[
\prod_{B(x) < q \leq \sqrt{x}} {\biggl(1 + \frac{1}{q^{1+iu}} + O \biggl( \frac{1}{q^{3/2}} \biggr) \biggr)} . 
\]
\end{lemma}
\begin{proof}
Assuming that the lemma holds with $\prod_{B(x) < q \leq \sqrt{x}} {\bigl(1 - \frac{1}{q^{1+iu}} \bigr)}^{-1} $, the last statement follows from the elementary estimate
\begin{align*}
\prod_{B(x) < q \leq \sqrt{x}} {\biggl(1 + \frac{1}{q^{1+iu}} + O \biggl( \frac{1}{q^{3/2}} \biggr) \biggr)} = \biggl( 1 + O \biggl( \frac{1}{{B(x)}^{1/2}} \biggr) \biggr) \prod_{B(x) < q \leq \sqrt{x}} {\biggl(1 - \frac{1}{q^{1+iu}} \biggr) }^{-1}.
\end{align*}
To prove the first statement, we notice that the integral in the lemma is approximately (the average of) a count of rough numbers: numbers whose prime factors are all larger than some threshold. To show this, we first introduce a weight $\frac{1}{1 + iu}$. By our conditions on $W(x)$, we have $\frac{1}{1 + iu} = 1 + O \bigl( \frac{1}{{(\log x)}^{2/3}} \bigr)$. Hence the left-hand side of the first equation in Lemma~\ref{l:roughintest} is
\[ 
\int_{|u| \leq \frac{W(x)}{\log x}} \prod_{B(x) < q \leq \sqrt{x}} {\biggl(1 - \frac{1}{q^{1+iu}} \biggr)}^{-1} \sum_{\sqrt{x} < p \leq x} \frac{x^{iu}}{p^{1+iu}} \frac{du}{1+iu} + O \biggl( \frac{1}{{(\log x)}^{4/3}} \exp \biggl( \sum_{B(x) < q \leq \sqrt{x}} \frac{1}{q} \biggr) \biggr).
\]
The error term here is $\ll \frac{1}{{(\log x)}^{1/3} \log B(x)}$.
We now extend the range of integration by showing that the integral over $\frac{W(x)}{\log x} \leq |u| \leq {(\log x)}^2$ is small. Seeing as we will break up the integral and apply the triangle inequality, we can obtain an upper bound by assuming that $W(x) = {(\log \log x)}^2$. We split the integral over $|u| \in [\frac{W(x)}{\log x}, {(\log x)}^2]$ into $\ll \log \log x$ dyadic ranges, and note that, for $M \in [\frac{{(\log \log x)}^2}{\log x} , {(\log x)}^2]$, we have 
\begin{align*} 
& \int_M^{2M} \prod_{B(x) < q \leq \sqrt{x}} {\biggl(1 - \frac{1}{q^{1+iu}} \biggr)}^{-1} \sum_{\sqrt{x} < p \leq x} \frac{x^{iu}}{p^{1+iu}} \frac{du}{1+iu} \\ 
\ll & \frac{1}{M \log x} \int_M^{2M} \biggl| \prod_{B(x) < q \leq x^{1/{(\log \log x)}^2}} {\biggl(1 - \frac{1}{q^{1+iu}} \biggr)}^{-1} \prod_{x^{1/{(\log \log x)}^2} \leq q \leq \sqrt{x}} {\biggl(1 - \frac{1}{q^{1+iu}} \biggr)}^{-1} \biggr| \, du \\
\ll & \frac{1}{M \log x}\int_M^{2M} \frac{\log x}{{(\log \log x)}^2 \log B(x)} \exp \biggl( \frac{{(\log \log x)}^2}{M \log x} \biggr) \, du \ll \frac{1}{{(\log \log x)}^2 \log B(x)},
\end{align*}
where to obtain the second and third lines we have applied Lemma~\ref{l:pwprimebound} in addition to standard prime number estimates. An identical bound holds over the negative range, so that adding up the dyadic blocks, the contribution from $\frac{{(\log \log x)}^2}{\log x} \leq |u| \leq {(\log x)}^2$ is $\ll \frac{1}{(\log \log x) \log B(x)}$. Therefore, the integral in our lemma is equal to
\[ 
\int_{|u| \leq {(\log x)}^2} \prod_{B(x) < q \leq \sqrt{x}} {\biggl(1 - \frac{1}{q^{1+iu}} \biggr)}^{-1} \sum_{\sqrt{x} < p \leq x} \frac{x^{iu}}{p^{1+iu}} \frac{du}{1+iu} + O \biggl( \frac{1}{(\log \log x)\log B(x)} \biggr) .
\]
Rearranging, this is
\[ 
\frac{2 \pi}{x} \sum_{\sqrt{x} < p \leq x} \frac{1}{2 \pi i} \int_{1 - i {(\log x)}^2}^{1 + i {(\log x)}^2} \prod_{B(x) < q \leq \sqrt{x}} {\biggl(1 - \frac{1}{q^{s}} \biggr)}^{-1} {\biggl( \frac{x}{p} \biggr)}^{s} \frac{ds}{s} + O \biggl( \frac{1}{(\log \log x)\log B(x)} \biggr) .
\]
By~\citet[Theorem~5.2 and Corollary~5.3]{MVmultnt}, we find that, for $\sqrt{x} < p \leq x$, we have
\begin{multline*} 
\frac{1}{2 \pi i} \int_{1 - i {(\log x)}^2}^{1 + i {(\log x)}^2} \prod_{B(x) < q \leq \sqrt{x}} {\biggl(1 - \frac{1}{q^{s}} \biggr)}^{-1} {\biggl( \frac{x}{p} \biggr)}^{s} \frac{ds}{s} = \sum_{\substack{n \leq x/p \\ q|n \Rightarrow q > B(x)}} 1 \\ 
+ O \biggl( 1 + \sum_{\substack{x/2p < n \leq 2x/p \\ n \neq x/p}} \min \biggl\{ 1, \frac{x}{p {(\log x)}^2 |x/p - n|} \biggr\} + \frac{x}{p {(\log x)}^2} \sum_{n \colon p|n \Rightarrow p \in ( B(x), \sqrt{x}) } \frac{1}{n} \biggr) .
\end{multline*}
Using the fact that $\frac{\log x}{{(\log \log x)}^{1000}} \ll \log B(x)$, this error term is
\begin{align*}
& \ll 1 + \sum_{\substack{|n - x/p| \leq x/p {(\log x)}^2 \\ n \neq x/p}} 1 + \frac{x}{p {(\log x)}^2} \sum_{\substack{|n - x/p| > x/p {(\log x)}^2 \\ x/2p < n \leq 2x/p}} \frac{1}{|x/p - n|} + \frac{x {(\log \log x)}^{1000}}{p {(\log x)}^2} \\
& \ll 1 + \frac{x {(\log \log x)}^{1000}}{p {(\log x)}^2}.
\end{align*}
Summing this over $\sqrt{x} < p \leq x$ and dividing through by $x$, this error contributes $\ll \frac{1}{(\log \log x) \log B(x)}$. We deduce that the integral on the left-hand side of Lemma~\ref{l:roughintest} is equal to
\begin{align*}
& \frac{2\pi}{x} \sum_{\sqrt{x} < p \leq x} \sum_{\substack{n \leq x/p \\ q|n \Rightarrow q > B(x)}} 1 + O \biggl( \frac{1}{(\log \log x)\log B(x)} \biggr) \\
= & \frac{2\pi}{x} \sum_{\sqrt{x} < p \leq x} \Phi \biggl( \frac{x}{p}, B(x) \biggr) + O \biggl( \frac{1}{(\log \log x)\log B(x)} \biggr),
\end{align*}
where $\Phi(x,y) \coloneqq \big| \bigl\{ n \leq x \colon p|n \Rightarrow p > y \bigr\} \bigr|$ is the count of $y$-rough numbers below $x$. We have the standard sieve upper bound (see~\cite[Corollary~I.4.13]{TenenbaumAPNT}) $\Phi(x,y) \ll \frac{x}{\log y}$ whenever $y \leq x$. Applying this bound, the contribution to the main term from large primes $\frac{x}{{B(x)}^{2\log \log x}} < p \leq x$ is
\begin{multline*}
\frac{2 \pi}{x} \biggl( \sum_{\frac{x}{{B(x)}^{2 \log \log x}} < p \leq \frac{x}{B(x)}} \Phi \biggl( \frac{x}{p}, B(x) \biggr) + \sum_{\frac{x}{B(x)} < p \leq x} 1 \biggr) \\
\ll \frac{1}{\log B(x)} \sum_{\frac{x}{{B(x)}^{2 \log \log x}} < p \leq \frac{x}{B(x)}} \frac{1}{p} + \frac{1}{\log x} \ll \frac{\log \log x}{\log x} ,
\end{multline*}
which is $ \ll \frac{1}{(\log \log x) \log B(x)}$, and so is engulfed into the previous error term. Thus our integral is equal to
\begin{equation}\label{equ:intasav}
\frac{2\pi}{x} \sum_{\sqrt{x} < p \leq \frac{x}{{B(x)}^{2 \log \log x}}} \Phi \biggl( \frac{x}{p}, B(x) \biggr) + O \biggl( \frac{1}{(\log \log x) \log B(x)} \biggr) . 
\end{equation}
It follows from~\cite[Theorem~III.6.1]{TenenbaumAPNT} and Mertens' third theorem that the first term is
\[ 
\frac{2 \pi e^{- \gamma}}{\log B(x)} \sum_{\sqrt{x} < p \leq \frac{x}{{B(x)}^{2 \log \log x}}} \frac{1}{p} + O \Biggl( \frac{1}{{\bigl( \log B(x) \bigr)}^2} + \frac{1}{x} \sum_{\sqrt{x} < p \leq \frac{x}{{B(x)}^{2 \log \log x}}} \Psi(x/p, B(x)) \Biggr) ,
\]
where $\Psi(x, y) \coloneqq \# \{ n \leq x \colon P(n) \leq y \}$ is the count of y-smooth numbers below $x$. With $u = \log x / \log y$, we have the classical bound $\Psi(x, y) \ll x e^{-u/2} $ (see~\cite[Theorem~III.5.1]{TenenbaumAPNT}). Applying this bound and Mertens' second theorem,~\eqref{equ:intasav} is
\[
\frac{2 \pi e^{- \gamma} \log 2}{\log B(x)} \biggl( 1 + O \biggl( \frac{1}{\log \log x} \biggr) \biggr) + O \Biggl( \sum_{\sqrt{x} < p \leq \frac{x}{{B(x)}^{2 \log \log x}}} \frac{1}{p} e^{- \log (x/p) / 2 \log B(x)} \Biggr) .
\]
It also follows from Mertens' second theorem that the second ``big Oh'' term is
\[ 
\ll \sum_{\sqrt{x} < p \leq \frac{x}{{{B(x)}^{2 \log \log x}}}} \frac{1}{p \log x} \ll \frac{1}{\log x} \ll \frac{1}{(\log \log x) \log B(x)}.
\]
We conclude that
\[ 
\int_{|u| \leq \frac{W(x)}{\log x}} \prod_{B(x) < q \leq \sqrt{x}} {\biggl(1 - \frac{1}{q^{1+iu}} \biggr)}^{-1} \sum_{\sqrt{x} < p \leq x} \frac{x^{iu}}{p^{1+iu}} \, du = \frac{2 \pi e^{-\gamma} \log 2}{\log B(x)} \biggl( 1 + O \biggl( \frac{1}{\log \log x} \biggr) \biggr) ,
\]
as required.
\end{proof}

\section{Preliminary results on random multiplicative functions}\label{s:probtools}
\subsection{Expectation results} 
We begin this section by collecting some standard results on random multiplicative functions. We also introduce a new result, Lemma~\ref{l:largeprimeav}, which we use to prove one of our key results: Proposition~\ref{p:smallprimesconc}.
\begin{lemma}[Expectation of sums]\label{l:expectation}
Let $f$ be a Rademacher or Steinhaus random multiplicative function. For any $k \in \N$, $a_n \in \C$, we have 
\[ 
\E \Bigl| \sum_{n \leq N} a_n f(n) \Bigr|^{2k} \leq {\biggl( \sum_{n \leq N} \tau_{2k-1} (n) |a_n|^2 \biggr)}^k , 
\]
where $\tau_k$ denotes the $k$-divisor function, $\tau_k (n) = \# \{(b_1,\ldots,b_k) \colon b_1 b_2 \ldots b_k = n, \, b_i \in \N \}$.
\end{lemma}
\begin{proof}
This is a case of Probability Result 2.3 of~\citet{HarperHM}.
\end{proof}
\begin{lemma}[Steinhaus Euler product expectation]\label{l:epstein}
If $f$ is a Steinhaus random multiplicative function, then for any real $300 \leq x \leq y$, say, and any $t , t' \in \R$, we have
\begin{multline*} 
\E \prod_{x < p \leq y} \biggl| 1 - \frac{f(p)}{p^{1/2 + i t}} \biggr|^{-1} \biggl| 1 - \frac{f(p)}{p^{1/2 + i t'}} \biggr|^{-1} \\ 
\ll \sqrt{\frac{\log y}{\log x} \biggl( 1 + \min \biggl\{ \frac{\log y}{\log x}, \frac{1}{|t - t'| \log x} + \frac{\log^2 \bigl(2 + |t - t'|\bigr)}{\log x} \biggr\} \biggr)},
\end{multline*}
and
\begin{multline*} 
\E \prod_{x < p \leq y} \biggl| 1 - \frac{f(p)}{p^{1/2 + i t}} \biggr|^{-2} \biggl| 1 - \frac{f(p)}{p^{1/2 + i t'}} \biggr|^{-2} \\ 
\ll {\Biggl( \min \biggl\{ \frac{\log y}{\log x} , 1 + \frac{{\bigl( | t - t' | \bigr)}^{1/100}}{\log x} \biggr\} \Biggr)}^4 {\biggl( \frac{\log y}{\log x} \biggr)}^2 {\biggl( 1 + \min \biggl\{ \frac{\log y}{\log x} , \frac{1}{|t - t'| \log x} \biggr\} \biggr)}^2 .
\end{multline*}
\end{lemma}
\begin{proof}
This is Euler Product Result 1 of~\cite{HarperLargeFluct} followed by a special case of Euler Product Result 2.1 of~\cite{HarperHM}, where we replaced the term $e^{O (1 + |t - t'|/{(\log x)}^{100})}$ there by the first term on the right-hand side. This can be done by the same considerations as in Euler Product Result 2.2 of~\cite{HarperHM}.
\end{proof}
\begin{lemma}[Rademacher Euler product expectation]\label{l:eprad}
If $f$ is a Rademacher random multiplicative function, then for any real $300 \leq x \leq y$, say, and any $t , t' \in \R$, we have
\begin{multline*}
\E \prod_{x < p \leq y} \biggl| 1 + \frac{f(p)}{p^{1/2 + i t}} \biggr| \biggl| 1 + \frac{f(p)}{p^{1/2 + i t'}} \biggr| \\
\ll \sqrt{\frac{\log y}{\log x} \biggl( 1 + \min \biggl\{ \frac{\log y}{\log x}, \frac{1}{|t - t'| \log x} + \frac{1}{|t + t'| \log x} + \frac{\log^2 \bigl(2 + |t| + |t'|\bigr)}{\log x} \biggr\} \biggr)},
\end{multline*}
and 
\begin{multline*}
\E \prod_{x < p \leq y} \biggl| 1 + \frac{f(p)}{p^{1/2 + i t}} \biggr|^2 \biggl| 1 + \frac{f(p)}{p^{1/2 + i t'}} \biggr|^2 \ll {\Biggl( \min \biggl\{ \frac{\log y}{\log x} , 1 + \frac{{\bigl( | t | + | t' | \bigr)}^{1/100}}{\log x} \biggr\} \Biggr)}^4 \\ 
\times {\biggl( \frac{\log y}{\log x} \biggr)}^2 {\biggl( 1 + \min \biggl\{ \frac{\log y}{\log x} , \frac{1}{|t + t'| \log x} \biggr\} \biggr)}^2 {\biggl( 1 + \min \biggl\{ \frac{\log y}{\log x} , \frac{1}{|t - t'| \log x} \biggr\} \biggr)}^2 .
\end{multline*}
\end{lemma}
\begin{proof}
This is Euler Product Result 2 of~\cite{HarperLargeFluct} followed by a special case of Euler Product Result 2.2 of~\cite{HarperHM}.
\end{proof}
\begin{lemma}[Splitting of random Euler products on large primes]\label{l:largeprimeav}
For any real $200 \leq x \leq y$, let $I(t) \coloneqq \prod_{x < p \leq y} {\bigl( 1 - \frac{f(p)}{p^{1/2 + it}} \bigr)}^{-1} $ in the Steinhaus case and $ I(t) \coloneqq \prod_{x < p \leq y} \bigl( 1 + \frac{f(p)}{p^{1/2 + it}} \bigr) $ in Rademacher case. In the Steinhaus case, we have 
\begin{multline*}
\E I(t_1) \conj{I(t_2)} \conj{I(t_3)} I(t_4) = \E I(t_1) \conj{I(t_2)} \E \conj{I(t_3)} I(t_4) \\ 
+ O \Biggl( {\biggl( \frac{\log y}{\log x} \biggr)}^2 \biggl| \exp \biggl( \sum_{x < p \leq y} \frac{1}{p^{1+i(t_1 - t_3)}} + \frac{1}{p^{1 + i(t_4 - t_2)}} + O \Bigl( \frac{1}{p^{3/2}} \Bigr) \biggr) - 1 \biggr| \Biggr),
\end{multline*}
and in the Rademacher case, we have
\begin{multline*}
\E I(t_1) \conj{I(t_2)} \conj{I(t_3)} I(t_4) = \E I(t_1) \conj{I(t_2)} \E \conj{I(t_3)} I(t_4) \\ 
+ O \Biggl( {\biggl( \frac{\log y}{\log x} \biggr)}^2 \biggl| \exp \biggl( \sum_{x < p \leq y} \frac{1}{p^{1 + i (t_1 - t_3)}} + \frac{1}{p^{1 + i (t_4 - t_2)}} + \frac{1}{p^{1 + i (t_1 + t_4)}} + \frac{1}{p^{1 - i(t_2 + t_3)}} + O \Bigl( \frac{1}{p^{3/2}} \Bigr) \biggr) - 1 \biggr| \Biggr),
\end{multline*}
Furthermore, if $\frac{1}{\log x} \leq |t_1 - t_3|, |t_2 - t_4|, |t_1 + t_4|, |t_2 + t_3| \leq {(\log x)}^{1000}$, then we have
\begin{multline*} 
\E I(t_1) \conj{I(t_2)} \conj{I(t_3)} I(t_4) = \E I(t_1) \conj{I(t_2)} \E \conj{I(t_3)} I(t_4) \\ 
+ O \Biggl( {\biggl( \frac{\log y}{\log x} \biggr)}^2 \biggl( \frac{1}{\log x} \biggl( \frac{1}{|t_1 - t_3|} + \frac{1}{|t_2 - t_4|} + \frac{1}{|t_1 + t_4|} + \frac{1}{|t_2 + t_3|} \biggr) \biggr) \Biggr), \end{multline*}
for both Steinhaus and Rademacher random multiplicative functions $f$.
\end{lemma}
\begin{proof}
The last statement follows from the first two statements by an application of Lemma~\ref{l:pwprimebound} and the fact that $\sum_{x < p \leq y} \frac{1}{p^{3/2}} \ll \frac{1}{\sqrt{x} \log x} $, which is negligible compared to the other terms given our conditions. We then apply the fact that for bounded $z \in \C$, we have $| \exp (z) - 1 | \ll |z| $. \\
For the first two statements, we begin by calculating the first terms on the right-hand sides, $\E I(t_1) \conj{I(t_2)} \E \conj{I(t_3)} I(t_4)$. For any $t, t' \in \R$, in the Steinhaus case, we have
\begin{multline*}
\E I(t) \conj{I(t')}
= \E \prod_{x < p \leq y} {\biggl( 1 - \frac{f(p)}{p^{1/2 + it}} - \frac{\conj{f(p)}}{p^{1/2 - it'}} + \frac{1}{p^{1+i(t-t')}} \biggr)}^{-1} \\
= \E \prod_{x < p \leq y} \biggl( 1 + \frac{f(p)}{p^{1/2 + it}} + \frac{\conj{f(p)}}{p^{1/2 - it'}} - \frac{1}{p^{1+i(t-t')}} + {\biggl( \frac{f(p)}{p^{1/2 + it}} + \frac{\conj{f(p)}}{p^{1/2 - it'}} - \frac{1}{p^{1+i(t-t')}} \biggr)}^2 + O\biggl( \frac{1}{p^{3/2}} \biggr) \biggr),
\end{multline*}
by expanding out the geometric series. Bringing the expectation inside the product, we find
\[ 
\E I(t) \conj{I(t')} = \prod_{x < p \leq y} \biggl( 1 + \frac{1}{p^{1+i(t-t')}} + O\biggl( \frac{1}{p^{3/2}} \biggr) \biggr). 
\]
In the Rademacher case, we obtain $\E I(t) \conj{I(t')} = \prod_{x < p \leq y} \bigl( 1 + \frac{1}{p^{1+i(t-t')}} \bigr)$ by similar techniques. Therefore, for $t_1, t_2, t_4, t_3 \in \R$, in the Steinhaus case, we have
\[ 
\E I(t_1) \conj{I(t_2)} \E \conj{I(t_3)} I(t_4) = \prod_{x < p \leq y} \biggl( 1 + \frac{1}{p^{1 + i(t_1 - t_2)}} + \frac{1}{p^{1 + i(t_4 - t_3)}} + O \biggl( \frac{1}{p^{3/2}} \biggr) \biggr), 
\]
which we note is equal to
\[
\prod_{x < p \leq y} {\biggl( 1 - \frac{1}{p^{1 + i(t_1 - t_2)}} - \frac{1}{p^{1 + i(t_4 - t_3)}} + O \biggl( \frac{1}{p^{3/2}} \biggr) \biggr)}^{-1}.
\]
Writing it in this way will aid with the calculations that follow. Note that the same result holds in the Rademacher case. We now calculate the left-hand side, $\E I(t_1) \conj{I(t_2)} \conj{I(t_3)} I(t_4) $, for any $t_1, t_2, t_4, t_3 \in \R$. In the Steinhaus case, this is
\begin{multline*} 
\E I(t_1) \conj{I(t_2)} \conj{I(t_3)} I(t_4) =
\E \prod_{x < p \leq y} \biggl( 1 - \frac{f(p)}{p^{1/2 + it_1}} - \frac{\conj{f(p)}}{p^{1/2 - it_2}} - \frac{f(p)}{p^{1/2 + it_4}} - \frac{\conj{f(p)}}{p^{1/2 - it_3}} \\ 
+ \frac{1}{p^{1 + i (t_1 - t_2)}} + \frac{1}{p^{1 + i (t_1 - t_3)}} + \frac{1}{p^{1 + i (t_4 - t_2)}} + \frac{1}{p^{1 + i (t_4 - t_3)}} + \frac{{f(p)}^2}{p^{1 + i(t_1 + t_4)}} + \frac{\conj{f(p)}^2}{p^{1 - i(t_2 + t_3)}} + O \biggl( \frac{1}{p^{3/2}} \biggr) \biggr)^{-1}. \end{multline*}
Expanding as a geometric series and bringing the expectation inside the product, \\ 
$\E I(t_1) \conj{I(t_2)} \conj{I(t_3)} I(t_4) $ is equal to
\[
\prod_{x < p \leq y} \biggl( 1 + \frac{1}{p^{1 + i (t_1 - t_2)}} + \frac{1}{p^{1 + i (t_1 - t_3)}} + \frac{1}{p^{1 + i (t_4 - t_2)}} + \frac{1}{p^{1 + i (t_4 - t_3)}} + O \biggl( \frac{1}{p^{3/2}} \biggr) \biggr).
\]
In the Rademacher case, one finds that $\E I(t_1) \conj{I(t_2)} \conj{I(t_3)} I(t_4)$ is equal to
\[
\prod_{x < p \leq y} \biggl( 1 + \frac{1}{p^{1 + i (t_1 - t_2)}} + \frac{1}{p^{1 + i (t_1 - t_3)}} + \frac{1}{p^{1 + i (t_4 - t_2)}} + \frac{1}{p^{1 + i (t_4 - t_3)}} + \frac{1}{p^{1 + i (t_1 + t_4)}} + \frac{1}{p^{1 - i(t_2 + t_3)}} + O \biggl( \frac{1}{p^2} \biggr) \biggr). 
\]
We now proceed by calculating $\bigl| \E I(t_1) \conj{I(t_2)} \conj{I(t_3)} I(t_4) - \E I(t_1) \conj{I(t_2)} \E \conj{I(t_3)} I(t_4) \bigr|$. Seeing as the latter term is nonzero, and by an application of Cauchy--Schwarz, we can write this difference as
\begin{align*} 
& \bigl| \E I(t_1) \conj{I(t_2)} \E \conj{I(t_3)} I(t_4) \bigr| \biggl| \frac{\E I(t_1) \conj{I(t_2)} \conj{I(t_3)} I(t_4)}{\E I(t_1) \conj{I(t_2)} \E \conj{I(t_3)} I(t_4)} - 1 \biggr| 
\\ \leq & {\Bigl( \E |I(t_1)|^2 \E |I(t_2)|^2 \E |I(t_4)|^2 \E |I(t_3)|^2 \Bigr)}^{1/2} \biggl| \frac{\E I(t_1) \conj{I(t_2)} \conj{I(t_3)} I(t_4)}{\E I(t_1) \conj{I(t_2)} \E \conj{I(t_3)} I(t_4)} - 1 \biggr| ,
\end{align*}
hence
\begin{equation}\label{equ:eptimesquotient}
\bigl| \E I(t_1) \conj{I(t_2)} \conj{I(t_3)} I(t_4) - \E I(t_1) \conj{I(t_2)} \E \conj{I(t_3)} I(t_4) \bigr| \ll {\biggl(\frac{\log y}{\log x}\biggr)}^2 \biggl| \frac{\E I(t_1) \conj{I(t_2)} \conj{I(t_3)} I(t_4)}{\E I(t_1) \conj{I(t_2)} \E \conj{I(t_3)} I(t_4)} - 1 \biggr|.
\end{equation}
where we have applied Lemmas~\ref{l:epstein} and~\ref{l:eprad} to evaluate the mean squares of the Euler products. Making use of our previous calculations, we find that, in the Steinhaus case,
\begin{align*}
\frac{\E I(t_1) \conj{I(t_2)} \conj{I(t_3)} I(t_4)}{\E I(t_1) \conj{I(t_2)} \E \conj{I(t_3)} I(t_4)} =& 
\prod_{x < p \leq y} \biggl( 1 + \frac{1}{p^{1 + i (t_1 - t_2)}} + \frac{1}{p^{1 + i (t_1 - t_3)}} + \frac{1}{p^{1 + i (t_4 - t_2)}} + \frac{1}{p^{1 + i (t_4 - t_3)}} \\ 
& + O \Bigl( \frac{1}{p^{3/2}} \Bigr) \biggr) \biggl( 1 - \frac{1}{p^{1 + i(t_1 - t_2)}} - \frac{1}{p^{1 + i(t_4 - t_3)}} + O \Bigl( \frac{1}{p^{3/2}} \Bigr) \biggr) \\
=& \prod_{x < p \leq y} \biggl( 1 + \frac{1}{p^{1 + i (t_1 - t_3)}} + \frac{1}{p^{1 + i (t_4 - t_2)}} + O \Bigl( \frac{1}{p^{3/2}} \Bigr) \biggl) .
\end{align*}
In the Rademacher case, one instead finds that
\[ 
\frac{\E I(t_1) \conj{I(t_2)} \conj{I(t_3)} I(t_4)}{\E I(t_1) \conj{I(t_2)} \E \conj{I(t_3)} I(t_4)} = \prod_{x < p \leq y} \biggl( 1 + \frac{1}{p^{1 + i (t_1 - t_3)}} + \frac{1}{p^{1 + i (t_4 - t_2)}} + \frac{1}{p^{1 + i (t_1 + t_4)}} + \frac{1}{p^{1 - i(t_2 + t_3)}} + O \Bigl( \frac{1}{p^{3/2}} \Bigr) \biggl). 
\]
We can then write these products as the exponential of a sum of logarithms. Taylor expanding these logarithms and combining with~\eqref{equ:eptimesquotient} gives the desired result.
\end{proof}
\subsection{Multiplicative Chaos Results}
Recall that $F_j (s) = \prod_{p \leq \sqrt{x}^{e^{-j}}} {\bigl( 1 - \frac{f(p)}{p^s} \bigr)}^{-1}$ in the Steinhaus case and $F_j (s) = \prod_{p \leq \sqrt{x}^{e^{-j}}} {\bigl( 1 + \frac{f(p)}{p^s} \bigr)}$ in the Rademacher case. We write $F$ in place of $F_0$. This subsection contains two key results about the typical behaviour of random Euler products. The first result involves a barrier event, $\mathcal{D}(t)$, which restricts the probability space to the likely event that the Euler product grows ``as expected''. By restricting to this event we prevent higher moments of $|F(1/2 + it)|$, which are dominated by increasingly unlikely large values of $|F(1/2 + it)|$, from growing uncontrollably. We make use of barrier events in the proof of Proposition~\ref{p:smallprimesconc}, where we encounter joint moment computations. The reader is encouraged to consult~\cite[Section~3]{HarperMomentsIII} for a more in-depth explanation. \\
To state the first result, we define the net of points $t(j)$ that serve as approximations for $t$ by $t(-1) = t$, and
\[ 
t(j) \coloneqq \max \biggl\{ u \leq t(j-1) \colon u = \frac{n}{((\log x)/e^j) \log((\log x)/e^j)} \text{ for some } n \in \Z \biggr\}, 
\]
for $0 \leq j \leq \log \log x - \log 2 -1$. We then have the following result of~\citet{HarperLargeFluct}:
\begin{lemma}[Expectation with barrier events]\label{l:mcbarrier}
Let $f$ be a Steinhaus or a Rademacher random multiplicative function. Suppose that $y \leq 1000 \log \log \log x$ is a non-negative integer. Then uniformly for all large $x$ and all $\frac{1}{{(\log \log x)}^{1000}} \leq |t| \leq {(\log \log x)}^{1000} $, the following is true: if we let $\mathcal{D}(t)$ denote the event that
\[ 
|F_j (1/2 + it(j)) | \leq \frac{\log x}{e^{j}} {(\log \log x)}^5 \quad \forall \, y \leq j \leq \log \log x - \log 2 - 1, 
\]
and let $\mathcal{A}(t)$ denote the event that 
\begin{align*}
&|F_j (1/2 + it)| \leq \frac{\log x}{e^j {(\log \log x)}^{1000}} \, \, \forall \, y \leq j \leq 0.99 \log \log x, \text{and } \\ 
&|F_j (1/2 + it)| \leq \frac{\log x}{e^j} {(\log \log x)}^6 \, \, \forall \, j \geq y, 
\end{align*}
then we have 
\[ 
\E |F_y (1/2 + it)|^2 \mathbf{1}_{\mathcal{D}(t)} \mathbf{1}_{\mathcal{A}(t) \text{ fails}} \ll \frac{\log x}{e^y} \frac{{(\log \log \log x)}^7}{\log \log x} . 
\]
Furthermore, the probability that $\mathcal{D}(t)$ holds for all $|t| \leq {(\log \log x)}^4$ is $1 - O \bigl({(\log \log x)}^{-4}\bigr)$.
\end{lemma}
\begin{proof}
The first part of this statement follows from Multiplicative Chaos Result 2 of~\citet{HarperLargeFluct}. For simplicity later on, we take similar $\mathcal{D} (t)$ and $\mathcal{A}(t)$ to those that appear in the statement there. Note that, in order to apply the result as stated there to the shorter Euler product $F_y (1/2 + it)$, one needs, for example, the event $\mathcal{D}(t)$ to imply $|F_j (1/2 + it(j)) | \leq \frac{\log x}{e^j} {(\log \log \sqrt{x}^{e^{-y}})}^5$ $\forall j \geq y$. However, since $y$ is small, $\log \log x \leq 2 \log \log \sqrt{x}^{e^{-y}}$, say, when $x$ is large, and an identical proof goes through if one only changes the bounds in our events by constants. The last statement on the probability of $\mathcal{D}(t)$ follows from a union bound, seeing as the probability that $\mathcal{D}(t)$ fails for some $|t| \leq {(\log \log x)}^4$ is
\begin{align*} 
\leq & \sum_{y \leq j \leq \log \log x - \log 2 -1} \sum_{|t(j)| \leq {(\log \log x)}^4} \p \Bigl( |F_j (1/2 + it(j)) | > \frac{\log x}{e^j} {(\log \log x)}^5 \Bigr) \\
\ll & \sum_{y \leq j \leq \log \log x - \log 2 -1} \biggl( \frac{(\log x){(\log \log x)}^5}{e^j} \biggr) \frac{e^j}{(\log x) {(\log \log x)}^{10}} \ll \frac{1}{{(\log \log x)}^4}.
\end{align*}
To obtain the second line, we have used the fact that the number of $|t(j)| \leq {(\log \log x)}^4$ is $\ll {(\log \log x)}^5 \frac{\log x}{e^j}$, applied Markov's inequality with second moments, and used the fact that $\E |F_j (1/2 + it)|^2 \ll \frac{\log x}{e^j}$, which follows from Lemmas~\ref{l:epstein} and~\ref{l:eprad}.
\end{proof}
The second result in this section contains key results from~\cite{HarperLM} which we utilise at various points throughout the paper. Recall that $F_{(x)} (s) \coloneqq \prod_{p \leq x} {\bigl( 1 - \frac{f(p)}{p^{s}} \bigr)}^{-1}$ in the Steinhaus case. Analogously, we define $F_{(x)} (s) \coloneqq \prod_{p \leq x} {\bigl( 1 + \frac{f(p)}{p^{s}} \bigr)}$ in the Rademacher case. We then have the following:
\begin{lemma}\label{l:mcexpectation} 
Let $f$ be a Rademacher or Steinhaus random multiplicative function. For any $0 < q < 1$, we have
\[ 
\E {\biggl( \frac{{(\log \log x)}^{1/2}}{\log x} \int_{\R} \biggl| \frac{F_{(x)} (1/2 + it)}{1/2 + it} \biggr|^2 \, dt \biggr)}^{q} \ll_q 1 ,
\]
and
\[
\E {\biggl( \frac{{(\log \log x)}^{1/4}}{\sqrt{x}} \Bigl| \sum_{\substack{n \leq x \\ P(n) > \sqrt{x}}} f(n) \Bigr| \biggr)}^{2q} \ll_q 1 .
\]
\end{lemma}
\begin{proof}
By Hölder's inequality, it suffices to prove the statements for $2/3 < q < 1$. Breaking up the integral identically to the end of~\citet[Section~2.4]{HarperLM}, in the Steinhaus case, the first expectation is
\[
\ll \E {\biggl( \frac{{(\log \log x)}^{1/2}}{\log x} \int_{-1/2}^{1/2} |F_{(x)} (1/2 + it)|^2 \, dt \biggr)}^{q},
\]
and in the Rademacher case, it is
\[ 
\ll \max_{N \in \Z} {(1 + |N|)}^{-1/4} \E {\biggl( \frac{{(\log \log x)}^{1/2}}{\log x} \int_{N-1/2}^{N + 1/2} |F_{(x)} (1/2 + it)|^2 \, dt \biggr)}^{q}.
\]
In the Steinhaus case, the statement follows from the main result proved in the section ``Proof of the upper bound in Theorem 1, assuming Key Propositions 1 and 2'' of~\cite[Section~4.1]{HarperLM} (taking $k=0$). In the Rademacher case, the result follows similarly from~\cite[Section~4.4]{HarperLM}. The second statement follows by combining the first statement with the fact that 
\[
\E {\biggl( \frac{{(\log \log x)}^{1/4}}{\sqrt{x}} \Bigl| \sum_{\substack{n \leq x \\ P(n) > \sqrt{x}}} f(n) \Bigr| \biggr)}^{2q} \ll \E {\biggl( \frac{{(\log \log x)}^{1/2}}{\log x} \int_{\R} \biggl| \frac{F_{(\sqrt{x})} (1/2 + it)}{1/2 + it} \biggr|^2 \, dt \biggr)}^{q} + o(1) ,
\]
uniformly for $2/3 < q < 1$, in both the Steinhaus and Rademacher cases. This follows almost identically to the handling of the $k=0$ portion of the sum in~\citet[Propositions~1 and~2]{HarperLM}, making only minor adjustments coming from the fact that here we condition on ${( f(p) )}_{p \leq \sqrt{x}}$ as opposed to ${( f(p) )}_{p \leq x^{e^{-1}}}$.
\end{proof}

\section{Conditional normal approximation}\label{s:normalapprox}
In this section, we will prove that our partial sums follow a conditional Gaussian distribution in the Steinhaus case. Note that an analogous statement (with a real Gaussian) holds in the Rademacher case, and is simpler to prove, seeing as one can apply the Berry--Esseen theorem.
\begin{proposition}\label{p:steincna}
Let $f$ be a Steinhaus random multiplicative function, and let $\mathcal{Z}$ denote a complex Gaussian that is independent from ${\bigl( f(p) \bigr)}_{p \leq \sqrt{x}}$ with mean $0$ and covariance matrix $\frac{{(\log \log x)}^{1/2}}{x} \sum_{\sqrt{x} < p \leq x} \bigl| \sum_{m \leq x/p} f(m) \bigr|^2 I_2 $, where $I_2$ is the two-dimensional identity matrix. For any fixed $a, b \in \R$, let $E$ be the embedding of $(- \infty, a) \times (- \infty, b)$ in the complex plane, then
\[
\p \biggl( \frac{{(\log \log x)}^{1/4}}{\sqrt{x}} \sum_{\substack{n \leq x \\ P(n) > \sqrt{x}}} f(n) \in E \biggr) = \p ( \mathcal{Z} \in E ) + O (x^{-1/20}) .
\]
\end{proposition}
The main tool we will use to prove this proposition is (a reformulation of) Lemma 4.1 of~\citet{HarLam}.
\begin{lemma}[Normal approximation tool]\label{l:normalapprox} 
Let $f$ be a Steinhaus random multiplicative function. For $\mathcal{A}$ a finite set, and ${(a_p)}_{p \text{ prime}}$ a fixed sequence of complex numbers, let $W_1$ and $W_2$ be the real and imaginary parts of the sum 
\[ 
\sum_{x < p \leq y} f(p) a(p) 
\]
respectively, for positive reals $ x < y $ where $(x,y]$ contains at least one prime, and define the random vector $\mathcal{W} = {(W_1, W_2)}^T$. Let $\mathcal{Z}$ be a multivariate normal random vector with zero mean and the same covariance matrix as $\mathcal{W}$. For any $3$-times differentiable function $h \colon \R^2 \rightarrow \R$, we have
\[ 
| \E h(\mathcal{W}) - \E h(\mathcal{Z}) | \ll |h|_2 {\biggl( \sum_{x < p \leq y} {|a(p)|}^4 \biggr)}^{1/2} + |h|_3 \sum_{x < p \leq y} |a(p)|^3 , 
\]
where $|h|_2 = \sup_{c+d=2} \| \frac{\partial^2 h}{\partial x^c \partial y^d} \|_{\infty} $ and $|h|_3 = \sup_{c+d=3} \| \frac{\partial^3 h}{\partial x^c \partial y^d} \|_{\infty} $.
\end{lemma}
\begin{proof}
This follows from an application of Lemma 4.1 of~\citet{HarLam}. In that result, we take $\mathcal{A} = \{ 1, 2 \}$, and for each $a \in \mathcal{A}$ let $c_i (a)$ be supported only on primes, with $ c_p (1) = a(p) $ and $c_p (2) = -i a(p)$. We take $V_i$ to be independent Steinhaus random variables, $m = \pi (y) - \pi(x)$, and $K = \sqrt{\pi(y) - \pi (x)}$, assuming that both these latter terms are non-zero. Seeing as our sum has mean zero, the Gaussian approximation does too.
\end{proof}
\begin{proof}[Proof of Proposition~\ref{p:steincna}] 
Let $\tilde{\E}$ and $\tilde{\p}$ denote the probability measure in the Steinhaus case conditioned on ${\bigl( f(p) \bigr)}_{p \leq \sqrt{x}}$. We have
\[ 
\frac{{(\log \log x)}^{1/4}}{\sqrt{x}} \sum_{\substack{n \leq x \\ P(n) > \sqrt{x}}} f(n) = \frac{{(\log \log x)}^{1/4}}{\sqrt{x}} \sum_{\sqrt{x} < p \leq x} f(p) \sum_{m \leq x/p} f(m) .
\]
Conditioning on ${\bigl(f(p)\bigr)}_{p \leq \sqrt{x}}$, we apply Lemma~\ref{l:normalapprox} with
\[ 
a(p) = \frac{{(\log \log x)}^{1/4}}{\sqrt{x}} \sum_{m \leq x/p} f(m).
\] 
In the lemma, we then have $\mathcal{W} = {(W_1, W_2)}^T$, where $W_1$ and $W_2$ denote the real and imaginary parts of 
\[ 
\frac{{(\log \log x)}^{1/4}}{\sqrt{x}} \sum_{\sqrt{x} < p \leq x} f(p) \sum_{m \leq x/p} f(m) , 
\] 
respectively. One can calculate that the covariance matrix of the normal approximation $\mathcal{Z}$ is
\[ 
\mathrm{Cov}_{i,j} = 
\begin{cases} 
\frac{{(\log \log x)}^{1/2}}{2x} \sum_{\sqrt{x} < p \leq x} \Bigl| \sum_{m \leq x/p} f(m) \Bigr|^2 \, &, \text{ if } i = j \\
0 \, &, \text{ otherwise.}
\end{cases} 
\]
Therefore, the conditional distribution of the partial sums is approximated by a complex Gaussian distribution with variances as given in Proposition~\ref{p:steincna}\footnote{Note that a standard complex Gaussian is the complex embedding of a two-dimensional Gaussian with covariance matrix $\frac{1}{2} I_2$. This is why the factor of $\frac{1}{2}$ does not appear in Proposition~\ref{p:steincna}.}. Now, for any $3$-times differentiable function $h \colon \R^2 \rightarrow \R$, we have 
\begin{multline}\label{equ:happrox}
| \tilde{\E} h(\mathcal{W}) -\tilde{\E} h(\mathcal{Z}) | \ll \\ \frac{|h|_2 {(\log \log x)}^{1/2}}{x} \sqrt{\sum_{\sqrt{x} < p \leq x} \Bigl| \sum_{m \leq x/p} f(m) \Bigr|^4} + \frac{|h|_3 {(\log \log x)}^{3/4}}{x^{3/2}} \sum_{\sqrt{x} < p \leq x} \Bigl| \sum_{m \leq x/p} f(m) \Bigr|^3 . 
\end{multline}
We begin by taking $h$ to be a $3$-times differentiable function $h \colon \R^2 \rightarrow \R$ that is a good approximation to the indicator function for $ E = (-\infty, a) \times (-\infty, b) $, for fixed $a, b \in \R$. Similarly to~\cite[Lemma~4.3]{HarLam}, we define
\[ \varphi^+ (x) = 
\begin{cases}
1 \, &\text{, if } x \leq 0 \\
f(x) \, &\text{, if } x \in (0,\delta) \\
0 \, &\text{, if } x \geq \delta
\end{cases} ,
\hspace{1cm}
\varphi^- (x) =
\begin{cases}
1 \, &\text{, if } x \leq - \delta \\
g(x) \, &\text{, if } x \in (- \delta,0) \\
0 \, &\text{, if } x \geq 0
\end{cases} ,
\]
for some $\delta > 0$ chosen later, where $f,g$ take values in $[0,1]$ so that $\varphi^\pm$ are 3-times differentiable. These can be chosen in such a way so that we have the derivative bounds $\|\frac{d^i \varphi^{\pm}}{d x^i} \|_{\infty} \ll \frac{1}{\delta^i}$ that hold uniformly for $i \in \{ 1,2,3 \}$. Now, for $\mathbf{x} = {(x_1, x_2)}^T$, consider the majorising and minorising functions for the indicator function of the set $ E = (-\infty, a) \times (-\infty, b) $ given by $h^+ (\mathbf{x}) \coloneqq \varphi^+ (x_1 - a) \varphi^+ (x_2 - b)$ and $ h^- (\mathbf{x}) \coloneqq \varphi^- (x_1 - a) \varphi^- (x_2 - b) $, respectively. We have 
\[ 
\tilde{\E} (h^- (\mathcal{W})) \leq \tilde{\p} (\mathcal{W} \in E) \leq \tilde{\E} (h^+ (\mathcal{W})) ,
\]
and it follows that 
\[
| \tilde{\p} (\mathcal{W} \in E) - \tilde{\p} (\mathcal{Z} \in E) | \leq | \tilde{\E} ( h^- (\mathcal{W}) ) - \tilde{\p} (\mathcal{Z} \in E) | + | \tilde{\E} ( h^+ (\mathcal{W}) ) - \tilde{\p} (\mathcal{Z} \in E) |. 
\]
We will focus on the first term on the right-hand side, with the second term being handled similarly. By the triangle inequality,
\[ 
| \tilde{\E} ( h^- (\mathcal{W}) ) - \tilde{\p} ( \mathcal{Z} \in E ) | \leq | \tilde{\E} ( h^- (\mathcal{W}) ) - \tilde{\E} ( h^- (\mathcal{Z}) ) | + | \tilde{\E} (h^- (\mathcal{Z})) - \tilde{\p} (\mathcal{Z} \in E) |. 
\]
We bound the first term using~\eqref{equ:happrox}, and the second term using standard results on multivariate Gaussians. Note that the second term is maximised when $a,b$ are near $0$, since this is when the density is largest. One can deduce the uniform bound 
\begin{multline*}
|\tilde{\p} (\mathcal{W} \in E) - \tilde{\p} (\mathcal{Z} \in E) | \ll \frac{{(\log \log x)}^{1/2}}{\delta^2 x} \sqrt{\sum_{\sqrt{x} < p \leq x} \Bigl| \sum_{m \leq x/p} f(m) \Bigr|^4} \\ + \frac{{(\log \log x)}^{3/4}}{\delta^3 x^{3/2}} \sum_{\sqrt{x} < p \leq x} \Bigl| \sum_{m \leq x/p} f(m) \Bigr|^3 + \delta \sqrt {\frac{{(\log \log x)}^{1/2}}{x} \sum_{\sqrt{x} < p \leq x} \Bigl| \sum_{m \leq x/p} f(m) \Bigr|^2} .
\end{multline*}
We now take the expectation, noting that $\E \tilde{\p} (A) = \p (A)$ for any event $A$. Using Hölder's inequality and applying Lemma~\ref{l:expectation} and the fact that $\sum_{n \leq x} \tau_k (n) \ll_k x {(\log x)}^{k-1}$, we have
\[
\p (\mathcal{W} \in E) = \p (\mathcal{Z} \in E) + O \biggl( \frac{{(\log x)}^2}{x^{1/4} \delta^2} + \frac{{(\log x)}^3}{x^{1/4} \delta^3} + \delta {(\log \log x)}^{1/4} \biggr) .
\]
Taking $\delta = x^{-1/16}$, say, we certainly have
\[ 
\p (\mathcal{W} \in E) = \p (\mathcal{Z} \in E) + O ( x^{-1/20} ) . 
\]
This is precisely the statement of Proposition~\ref{p:steincna}.
\end{proof}
\section{Analysis of the variance: proof of Theorem~\ref{t:variance}}\label{s:varianceanalysis}
\subsection{Initial analysis}
In light of Proposition~\ref{p:steincna}, to understand the distribution of our partial sums, it suffices to understand the distribution of the variance of the Gaussian, which is equal to $\frac{{(\log \log x)}^{1/2}}{x} \sum_{\sqrt{x} < p \leq x} \bigl| \sum_{m \leq x/p} f(m) \bigr|^2 $. We begin by analysing the probable behaviour of this quantity. Throughout this section, we will make frequent use of the notation $F_j (s) = \prod_{p \leq \sqrt{x}^{e^{-j}}} {\bigl( 1 - \frac{f(p)}{p^s} \bigr)}^{-1}$ in the Steinhaus case and $F_j (s) = \prod_{p \leq \sqrt{x}^{e^{-j}}} {\bigl( 1 + \frac{f(p)}{p^s} \bigr)}$ in the Rademacher case, and we write $F$ in place of $F_0$.
\begin{proposition}\label{p:perron}
Define
\[ 
\mathcal{V}_1 (x) = \frac{{(\log \log x)}^{1/2}}{x} \sum_{\sqrt{x} < p \leq x} \bigl| \sum_{m \leq x/p} f(m) \bigr|^2,
\]
and define $\mathcal{V}_2 (x)$ as
\[ 
\frac{{(\log \log x)}^{1/2}}{4 \pi^2} \iint_{\substack{|t_1|, |t_2| \leq {(\log \log x)}^4 \\ |t_1 - t_2| \leq \frac{{(\log \log x)}^5}{\log x}}} \frac{F (1/2 + i t_1) \conj{F (1/2 + i t_2) }}{(1/2 + it_1)(1/2 - it_2)} \sum_{\sqrt{x} < p \leq x} \frac{x^{i(t_1 - t_2)}}{p^{1 + i (t_1 - t_2)}} d t_1 d t_2, 
\]
then in both the Steinhaus and the Rademacher cases, we have
\[
| \mathcal{V}_1 (x) - \mathcal{V}_2 (x) | \ll \frac{1}{\log \log x},
\]
with probability $1 - O \bigl( {(\log \log x)}^{-1} \bigr)$.
\end{proposition}
In the quantity $\mathcal{V}_2 (x)$, it can quite easily be shown that with high probability we have approximate equality of the Euler products on small primes.
\begin{proposition}\label{p:smallprimes}
Let $\mathcal{V}_2 (x)$ be the same as in Proposition~\ref{p:perron} and define $\mathcal{V}_3 (x)$ as
\[
\frac{{(\log \log x)}^{1/2}}{4 \pi^2} \iint_{\substack{|t_1|, |t_2| \leq {(\log \log x)}^4 \\ |t_1 - t_2| \leq \frac{{(\log \log x)}^5}{\log x}}} \frac{|F_{y} (1/2 + i t_1)|^2 I_y (t_1) \conj{I_y (t_2)}}{|1/2 + it_1|^2} \sum_{\sqrt{x} < p \leq x} \frac{x^{i(t_1 - t_2)}}{p^{1 + i (t_1 - t_2)}} d t_1 d t_2 ,
\]
where $y = \lfloor 100 \log \log \log x \rfloor$, and 
\[ 
I_y (t) \coloneqq \prod_{\sqrt{x}^{e^{-y}} < p \leq \sqrt{x}} {\Bigl( 1 - \frac{f(p)}{p^{1/2 + i t}} \Bigr)}^{-1}, \quad \text{and} \quad I_y (t) \coloneqq \prod_{\sqrt{x}^{e^{-y}}< p \leq \sqrt{x}} {\Bigl( 1 + \frac{f(p)}{p^{1/2 + i t}} \Bigr)} ,
\]
in the Steinhaus and Rademacher cases, respectively.
Then
\[ 
| \mathcal{V}_2 (x) - \mathcal{V}_3 (x) | \leq \frac{1}{{(\log \log x)}^{5}} , 
\]
with probability $1 - O \bigl({(\log \log x )}^{-5}\bigr)$.
\end{proposition}
\begin{proof}[Proof of Proposition~\ref{p:perron}]
Let $\sum_{n \leq y}'$ denote the sum up to $y$ where, if $y$ is an integer, the last term is counted with weight $1/2$. Note that 
\begin{multline*}
\frac{{(\log \log x)}^{1/2}}{x} \sum_{\sqrt{x} < p \leq x} \bigl| \sum_{m \leq x/p} f(m) \bigr|^2 = \frac{{(\log \log x)}^{1/2}}{x} \sum_{\sqrt{x} < p \leq x} \bigl| \sideset{}{'}\sum_{m \leq x/p} f(m) \bigr|^2 \\ 
+ O \biggl( \frac{{(\log \log x)}^{1/2}}{x} \sum_{\sqrt{x} < p \leq x} \bigl| \sum_{m \leq x/p} f(m) \bigr| + \frac{{(\log \log x)}^{1/2}}{\log x} \biggr),
\end{multline*}
The expectation of the first error term is $\leq \frac{{(\log \log x)}^{1/2}}{x} \sum_{\sqrt{x} < p \leq x} {( \E \bigl| \sum_{m \leq x/p} f(m) \bigr|^2 )}^{1/2} \ll \frac{{(\log \log x)}^{1/2}}{\log x}$, so an application of Markov's inequality with first moments allows one to find that
\begin{equation}\label{equ:modifiedperronsum}
\mathcal{V}_1 (x) = \frac{{(\log \log x)}^{1/2}}{x} \sum_{\sqrt{x} < p \leq x} \bigl| \sideset{}{'}\sum_{m \leq x/p} f(m) \bigr|^2 + O \biggl( \frac{1}{{(\log x)}^{1/3}} \biggr),
\end{equation}
with probability $1 - O({(\log x)}^{-1/2})$. We now apply Perron's formula to rewrite the modified sum. By~\citet[Corollary~5.3]{MVmultnt}, taking $T=x^{1/3}$, we find that $\sum_{m \leq x/p}' f(m) $ is equal to
\[
\frac{1}{2 \pi i} \int_{1-i x^{1/3}}^{1+i x^{1/3}} \frac{F (s) x^{s}}{s p^s} \, ds + O \Biggl( \sum_{\substack{x/2p < n \leq 2x/p \\ n \neq x/p}} \min \biggl\{ 1, \frac{x^{2/3}}{p|x/p - n|} \biggr\} + \frac{x^{2/3}}{p} \sum_{n \leq x/p} \frac{1}{n} \Biggr),
\]
uniformly for all primes $\sqrt{x} < p \leq x$. For $\alpha \in \R$, let $\| \alpha \|$ denote the distance from $\alpha$ to the nearest integer. A straightforward calculation shows that the ``big Oh'' term is $ \ll \mathbf{1}_{\| x/p \| \leq x^{-1/30}} + \frac{x^{7/10} \log x}{p}$. Shifting the contour to the left, we have
\[
\sideset{}{'}\sum_{m \leq x/p} f(m) = \frac{1}{2 \pi i} \int_{1/2-i x^{1/3}}^{1/2+i x^{1/3}} \frac{F (s) x^{s}}{s p^s} \, ds + O \bigl( E (x,p) \bigr),
\]
where
\[
E (x,p) = \mathbf{1}_{\| x/p \| \leq x^{-1/30}} + \frac{x^{7/10} \log x}{p} + \frac{1}{x^{1/3}} \int_{1/2}^{1} \bigl( |F (\sigma + i x^{1/3})| + |F (\sigma - i x^{1/3})| \bigr) {\biggl( \frac{x}{p} \biggr)}^{\sigma} \, d \sigma.
\]
From this, we deduce that $\frac{{(\log \log x)}^{1/2}}{x} \sum_{\sqrt{x} < p \leq x} \bigl| \sum_{m \leq x/p}' f(m) \bigr|^2 $ is equal to
\begin{equation}\label{equ:longperron}
\frac{{(\log \log x)}^{1/2}}{4 \pi^2} \iint_{|t_1|, |t_2| \leq x^{1/3}} \frac{F (1/2 + i t_1) \conj{F (1/2 + i t_2)}}{(1/2 + i t_1)(1/2 - i t_2)} x^{i (t_1 - t_2)} \sum_{\sqrt{x} < p \leq x} \frac{1}{p^{1+i(t_1-t_2)}} \, d t_1 \, d t_2 ,
\end{equation}
plus an error term of size
\[
\ll \frac{{(\log \log x)}^{1/2}}{x} \sum_{\sqrt{x} < p \leq x} \biggl( \frac{x^{1/2} E (x,p)}{p^{1/2}} \int_{1/2-i x^{1/3}}^{1/2+i x^{1/3}} \frac{|F (1/2 + it)|}{|1/2 + it|} \, dt + {E (x,p)}^2 \biggr)
\]
Using the facts that $\E |F (\sigma + ix^{1/3})|^2 \ll {\log x} $ and $\E |F(1/2 + i t)|^2 \ll {\log x} $ (see Euler Product Results~2.1 and~2.2 of~\cite{HarperHM}), one can apply Markov's inequality and Cauchy--Schwarz to find that this error is $\ll x^{-1/100}$ with probability $1 - O(x^{-1/100})$, say. From this and~\eqref{equ:modifiedperronsum}, we deduce that $\mathcal{V}_1 (x)$ is equal to~\eqref{equ:longperron} plus an error of size $\ll {(\log x)}^{-1/3}$, with probability $1 - O({(\log x)}^{-1/2})$. \\ 
The goal now is to reduce the range of integration in~\eqref{equ:longperron}. This follows similarly to~\citet[Section~3.2]{HarperLargeFluct}, yet we need to reduce the range more so that the Euler products over small primes are approximately equal over the entire range of integration. One would expect that, with high probability, the dominant contribution to the integral in~\eqref{equ:longperron} comes from points where $t_1$ and $t_2$ are relatively close. When this happens, the Euler products can reinforce one another and make a large contribution. Conversely, when $t_1$ and $t_2$ differ by some amount $\gg \frac{1}{\log x}$, the Euler products $F (1/2 + i t_1)$ and $\conj{F (1/2 + i t_2)}$ begin to decorrelate on large primes, giving a smaller contribution to the integral. Therefore, we should be able to reduce the domain of integration to something resembling (but maybe not quite) $ |t_1 - t_2 | \leq \frac{1}{\log x} $. We begin by removing the part of the integral where $\max \{ |t_1|, |t_2| \} > {(\log x)}^2$, again making use of the fact that $\E |F(1/2 + i t)|^2 \ll {\log x} $ in both the Steinhaus and Rademacher cases. By symmetry, the contribution from points where $\max \{ |t_1|, |t_2| \} > {(\log x)}^2$ is
\begin{multline*}
\ll {(\log \log x)}^{1/2} \int_{-x^{1/3}}^{x^{1/3}} \frac{|F (1/2 + it_1)|}{1 + |t_1|} \\
\Biggl( \int_{\max \{ |t_1|, {(\log x)}^2 \} \leq |t_2| \leq x^{1/3}} \frac{|F (1/2 + it_2)|}{1 + |t_2|} \biggl| \sum_{\sqrt{x} < p \leq x} \frac{1}{p^{1 + i(t_1 - t_2)}} \biggr| \, dt_2 \Biggr) \, dt_1 , 
\end{multline*}
the expectation of which is
\[ 
\ll (\log x) {(\log \log x)}^{1/2} \int_{-x^{1/3}}^{x^{1/3}} \frac{1}{1 + |t_1|} \int_{\max \{ |t_1|, {(\log x)}^2 \} \leq |t_2| \leq x^{1/3}} \frac{1}{1 + |t_2|} \biggl| \sum_{\sqrt{x} < p \leq x} \frac{1}{p^{1 + i(t_1 - t_2)}} \biggr| \, dt_2 \, dt_1 .
\]
By Cauchy--Schwarz and Lemma~\ref{l:meansqlargeprimes}, the innermost integral is 
\begin{align*}
\ll & {\Biggl( \int_{\max \{ |t_1|, {(\log x)}^2 \} \leq |t_2| \leq x^{1/3}} \frac{d t_2}{1 + |t_2|^2} \int_{|t_2| \leq x^{1/3}} \biggl| \sum_{\sqrt{x} < p \leq x} \frac{1}{p^{1 + i (t_1 - t_2)}} \biggr|^2 \, d t_2 \Biggr)}^{1/2} \\
\ll & {\Biggl( \frac{1}{\max \big\{ |t_1|, {(\log x)}^2 \big\}} \times \frac{1}{\log x} \Biggr)}^{1/2}.
\end{align*}
Now integrating over $t_1$, the expectation of the contribution to~\eqref{equ:longperron} from points where $\max \{ |t_1|, |t_2| \} > {(\log x)}^2$ is 
\[ 
\ll {\biggl( \frac{\log \log x}{\log x} \biggr)}^{1/2} \int_{-{(\log x)}^2}^{{(\log x)}^2} \frac{dt_1}{1 + |t_1|} + {(\log x)}^{1/2} {(\log \log x)}^{1/2} \int_{|t_1| > {(\log x)}^2} \frac{dt_1}{|t_1|^{3/2}} \ll \frac{{(\log \log x)}^{3/2}}{{(\log x)}^{1/2}} .
\]
We deduce that our variances, $\mathcal{V}_1(x)$, are equal to
\[
\frac{{(\log \log x)}^{1/2}}{4 \pi^2} \iint_{|t_1|, |t_2| \leq {(\log x)}^2} \frac{F (1/2 + i t_1) \conj{F (1/2 + i t_2)}}{(1/2 + i t_1)(1/2 - i t_2)} x^{i (t_1 - t_2)} \sum_{\sqrt{x} < p \leq x} \frac{1}{p^{1+i(t_1-t_2)}} \, d t_1 \, d t_2 ,
\]
up to an additive error of size $O ( {(\log x )}^{-1/10} )$, with probability $1 - O ( {(\log x)}^{-1/10} )$, in both the Steinhaus and Rademacher cases. We now reduce to the range $|t_1 - t_2| \leq \frac{{(\log \log x)}^{5}}{\log x}$, using the decorrelation of the Euler products to show that the complement of this range contributes relatively little, with high probability. Using Lemma~\ref{l:pwprimebound} to bound the prime number sum, and applying Lemmas~\ref{l:epstein} and~\ref{l:eprad} to bound the expectation of the Euler products, we calculate that the expected contribution to the integral from the range $|t_1 - t_2 | > \frac{{(\log \log x)}^5}{\log x}$ is
\begin{multline*}
\ll \frac{1}{{(\log x)}^{1/2}} \iint_{\substack{|t_1|, |t_2| \leq {(\log x)}^2 \\ |t_1 - t_2| > \frac{{(\log \log x)}^{5}}{\log x}}} \frac{1}{(1 + |t_1|)(1 + |t_2|)} \frac{1}{|t_1 - t_2|} \\
\times {\biggl( \frac{\mathbf{1}_{|t_1 - t_2| \leq 1}}{\sqrt{|t_1 - t_2|}} + \min \biggl\{ \sqrt{\log x}, \frac{\mathbf{1}_{|t_1 + t_2| \leq 1}}{\sqrt{|t_1 + t_2|}} \biggr\} + \log (2 + |t_1| + |t_2|) \biggr)} \, dt_1 \, dt_2 ,
\end{multline*}
in both the Steinhaus and the Rademacher cases. Note that we have added in the indicator functions for $|t_1 - t_2| \leq 1$ and $|t_1 + t_2| \leq 1$ since otherwise the last term in the parenthesis dominates. The contribution from the last term in the parenthesis is 
\[
\ll \frac{\log \log x}{{(\log x)}^{1/2}} \iint_{\substack{|t_1|, |t_2| \leq {(\log x)}^2 \\ |t_1 - t_2| > \frac{{(\log \log x)}^{5}}{\log x}}} \frac{1}{1 + |t_1|} \frac{1}{|t_1 - t_2|} \, dt_1 \, dt_2 \ll \frac{{(\log \log x)}^3}{{(\log x)}^{1/2}},
\]
and when $|t_1 - t_2| \leq 1$ or $|t_1 + t_2| \leq 1$, we have $1 + |t_1| \asymp 1 + |t_2|$, and so the contribution from the first two terms in the parenthesis is
\[
\ll \frac{1}{{(\log x)}^{1/2}} \iint_{\substack{|t_1|, |t_2| \leq {(\log x)}^2 \\ |t_1 - t_2| > \frac{{(\log \log x)}^{5}}{\log x}}} \biggl( \frac{\mathbf{1}_{|t_1 - t_2| \leq 1}}{\sqrt{|t_1 - t_2|}} + \min \biggl\{ \sqrt{\log x}, \frac{1}{\sqrt{|t_1 + t_2|}} \biggr\} \biggr) \frac{d t_1 \, d t_2}{{(1 + |t_1|)}^2 |t_1 - t_2|} .
\]
Making the substitution $t_2 = t_1 + u$, this is
\begin{multline*}
\ll \frac{1}{{(\log x)}^{1/2}} \int_{- {(\log x)}^2}^{{(\log x)}^2} \frac{1}{{(1 + |t_1|)}^2} \int_{\frac{{(\log \log x)}^5}{\log x}}^{2 (\log x)^2} \frac{du}{u^{3/2}} \, dt_1 \\
+ \frac{1}{{(\log x)}^{1/2}} \int_{- {(\log x)}^2}^{{(\log x)}^2} \frac{1}{{(1 + |t_1|)}^2} \int_{\frac{{(\log \log x)}^5}{\log x} \leq |u| \leq 2{(\log x)}^2} \min \biggl\{ \sqrt{\log x}, \frac{1}{\sqrt{|2t_1 + u|}} \biggr\} \, \frac{du}{u} \, d t_1 ,
\end{multline*}
The contribution from the first term can be shown to be $\frac{1}{{(\log \log x)}^{5/2}}$ by straightforward calculations, and by swapping the order of integration and applying Cauchy--Schwarz, the second term can be seen to satisfy the same bound. Therefore, with probability $1 - O({(\log \log x)}^{-1})$, our variances $\mathcal{V}_1 (x)$ are 
\begin{equation}\label{equ:var4}
\frac{{(\log \log x)}^{1/2}}{4 \pi^2} \iint_{\substack{|t_1|, |t_2| \leq {(\log x)}^2 \\ |t_1 - t_2| \leq \frac{{(\log \log x)}^{5}}{\log x}}} \frac{F (1/2 + i t_1) \conj{F (1/2 + i t_2)}}{(1/2 + i t_1)(1/2 - i t_2)} x^{i (t_1 - t_2)} \sum_{\sqrt{x} < p \leq x} \frac{1}{p^{1+i(t_1-t_2)}} \, d t_1 \, d t_2 , 
\end{equation}
plus an error of size $O ( {(\log \log x)}^{-1} )$, in both the Rademacher and Steinhaus cases. Finally, we further reduce the ranges of $|t_1|, |t_2|$. Using the expectation bound $\E | F (1/2 + i t_1) \conj{F (1/2 + i t_2)} | \ll \log x$, the expectation of the part of the integral where $\max \{|t_1|, |t_2|\} \geq {(\log \log x)}^{4}$ is, by symmetry,
\begin{multline*}
\ll (\log x) {(\log \log x)}^{1/2} \int_{-{(\log x)}^2}^{{(\log x)}^2} \frac{1}{{(1 + |t_1|)}^2} \\ 
\Biggl( \int_{\substack{\max \{ |t_1|, {(\log \log x)}^4 \} \leq |t_2| \leq {(\log x)}^2 \\ |t_1 - t_2 | \leq \frac{{(\log \log x)}^5}{\log x}}} \min \biggl\{ 1, \frac{1}{|t_1 - t_2| \log x} \biggr\} \, dt_2 \Biggr) \, dt_1 ,
\end{multline*}
where we have used Lemma~\ref{l:pwprimebound} and the fact that on the range of integration $1 + |t_1| \asymp 1 + |t_2|$. Seeing as $t_1$ and $t_2$ are close, this is
\begin{align*}
\ll & \log x {(\log \log x)}^{1/2} \int_{\frac{{(\log \log x)}^4}{2}}^{{(\log x)}^2} \frac{1}{|t_1|^2} \int_{|t_1 - t_2 | \leq \frac{{(\log \log x)}^5}{\log x}} \min \biggl\{ 1, \frac{1}{|t_1 - t_2| \log x} \biggr\} \, dt_1 \, dt_2 \\
\ll & {(\log \log x)}^{1/2} \int_{\frac{{(\log \log x)}^4}{2}}^{{(\log x)}^2} \frac{1}{|t_1|^2} \, dt_1 \\
& + {(\log \log x)}^{1/2} \int_{\frac{{(\log \log x)}^4}{2}}^{{(\log x)}^2} \frac{1}{|t_1|^2} \int_{\frac{1}{\log x} \leq |t_1 - t_2 | \leq \frac{{(\log \log x)}^5}{\log x}} \frac{1}{|t_1 - t_2|} \, dt_1 \, dt_2 \\ 
\ll & \frac{1}{{(\log \log x)}^{5/2}}.
\end{align*}
Therefore, an application of Markov's inequality allows us to say that the contribution to~\eqref{equ:var4} from points where $\max \{|t_1|, |t_2|\} \geq {(\log \log x)}^{4}$ is $\ll {(\log \log x)}^{-1}$, with probability $1 - O ({(\log \log x)}^{-1})$. Putting this all together, we have shown that $\mathcal{V}_1 (x) = \mathcal{V}_2 (x) + O ( {( \log \log x )}^{-1} )$ with probability $1 - O ( {( \log \log x )}^{-1} )$. This completes the proof of Proposition~\ref{p:perron}.
\end{proof}
\begin{proof}[Proof of Proposition~\ref{p:smallprimes}]
Let $y = \lfloor 100 \log \log \log x \rfloor$. To prove Proposition~\ref{p:smallprimes}, we need to show that 
\begin{multline}\label{equ:spl}
\p \Biggl( \biggl| \iint_{\substack{|t_1|, |t_2| \leq {(\log \log x)}^4 \\ |t_1 - t_2| \leq \frac{{(\log \log x)}^5}{\log x}}} \frac{F (1/2 + it_1) I_y (t_2)}{1/2 + i t_1} \Biggl( \frac{\conj{F_{y} (1/2 + it_1)}}{1/2 - i t_1} - \frac{\conj{F_{y} (1/2 + i t_2)}}{1/2 - i t_2} \Biggr) \\
\times \sum_{\sqrt{x} < p \leq x} \frac{x^{i(t_1 - t_2)}}{p^{1 + i (t_1 - t_2)}} \, d t_1 \, d t_2 \biggr| \geq \frac{1}{{(\log \log x)}^{5}} \Biggr) \ll \frac{1}{{(\log \log x)}^{5}},
\end{multline}
We can upper bound the left-hand side using Markov's inequality with first moments. This gives an upper bound of
\begin{multline*} 
{(\log \log x)}^{5} \E \biggl| \iint_{\substack{|t_1|, |t_2| \leq {(\log \log x)}^4 \\ |t_1 - t_2| \leq \frac{{(\log \log x)}^5}{\log x}}} \frac{F (1/2 + it_1) I_y (t_2)}{1/2 + i t_1} \Biggl( \frac{\conj{F_{y} (1/2 + it_1)}}{1/2 - i t_1} - \frac{\conj{F_{y} (1/2 + i t_2)}}{1/2 - i t_2} \Biggr) \\ 
\times \sum_{\sqrt{x} < p \leq x} \frac{x^{i(t_1 - t_2)}}{p^{1 + i (t_1 - t_2)}} \, d t_1 \, d t_2 \biggr|. 
\end{multline*}
By an application of Cauchy--Schwarz, this is smaller than a constant multiplied by
\begin{multline*} 
{(\log \log x)}^{5} \E \Biggl[ {\biggl( \iint_{\substack{|t_1|, |t_2| \leq {(\log \log x)}^4 \\ |t_1 - t_2| \leq \frac{{(\log \log x)}^5}{\log x}}} \biggl| \frac{F (1/2 + it_1)}{1/2 + i t_1} \biggr|^2 \, dt_1 \, dt_2 \Biggr)}^{1/2} \\ 
{\Biggl( \iint_{\substack{|t_1|, |t_2| \leq {(\log \log x)}^4 \\ |t_1 - t_2| \leq \frac{{(\log \log x)}^5}{\log x}}} \biggl| \frac{\conj{F_{y} (1/2 + it_1)}}{1/2 - i t_1} - \frac{\conj{F_{y} (1/2 + i t_2)}}{1/2 - i t_2} \biggr|^2 |I_y (t_2)|^2 \biggr| \, d t_1 \, d t_2 \Biggr)}^{1/2} \Biggr]. 
\end{multline*}
Applying Cauchy--Schwarz to the expectation followed by Lemmas~\ref{l:epstein} and~\ref{l:eprad}, we have the upper bound
\[ 
\ll {(\log \log x)}^{10} {\Biggl( \iint_{\substack{|t_1|, |t_2| \leq {(\log \log x)}^4 \\ |t_1 - t_2| \leq \frac{{(\log \log x)}^5}{\log x}}} \E \biggl| \frac{\conj{F_{y} (1/2 + it_1)}}{1/2 - i t_1} - \frac{\conj{F_{y} (1/2 + i t_2)}}{1/2 - i t_2} \biggr|^2 \E |I_y (t_2)|^2 \, d t_1 \, d t_2 \Biggr)}^{1/2} , 
\]
where we have used the fact that $I_y (t_2)$ is independent of the other Euler product term. It follows from Lemmas~\ref{l:epstein} and~\ref{l:eprad} that $\E |I_y (t_2)|^2 \ll {(\log \log x)}^{100}$. Since we also have $\frac{1}{1/2 - i t_2} = \frac{1}{1/2 - i t_1} + O \bigl( |t_1 - t_2| \bigr)$, the above display is
\begin{multline*} 
\ll {(\log \log x)}^{60} \Biggl( \iint_{\substack{|t_1|, |t_2| \leq {(\log \log x)}^4 \\ |t_1 - t_2| \leq \frac{{(\log \log x)}^5}{\log x}}} \frac{\E \bigl| \conj{F_{y} (1/2 + it_1)} - \conj{F_{y} (1/2 + i t_2)} \bigr|^2}{|1/2 + it_1|^2} \\ 
+ |t_1 - t_2|^2 \E \bigl| F_{y} (1/2 + it_2) \bigr|^2 \, d t_1 \, d t_2 \Biggr)^{1/2} . 
\end{multline*}
Seeing as $\E \bigl| F_{y} (1/2 + it_2) \bigr|^2 \ll \log x/{(\log \log x)}^{100}$, this is
\begin{equation}\label{equ:expofdif}
\ll {(\log \log x)}^{60} {\Biggl( \iint_{\substack{|t_1|, |t_2| \leq {(\log \log x)}^4 \\ |t_1 - t_2| \leq \frac{{(\log \log x)}^5}{\log x}}} \frac{\E \bigl| \conj{F_{y} (1/2 + it_1)} - \conj{F_{y} (1/2 + i t_2)} \bigr|^2}{|1/2 + it_1|^2} \, d t_1 \, d t_2 \Biggr)}^{1/2} + \frac{{(\log \log x)}^{50}}{\log x}.
\end{equation}
Recall that $y = \lfloor 100 \log \log \log x \rfloor$ and define $B(x) \coloneqq \sqrt{x}^{e^{-y}}$. Note that the sum $\sum_{n \colon P(n) \leq B(x)} \frac{1}{\sqrt{n}}$ is convergent, so we can write our Euler product difference as a sum to evaluate its expectation. Letting $M(x) = x^{1/{(\log \log x)}^{50}}$ and taking $\sigma > 0$, we apply Rankin's trick to find that
\begin{align*} 
\E \bigl|& \conj{F_{y} (1/2 + it_1)} - \conj{F_{y} (1/2 + i t_2)} \bigr|^2 \leq \sum_{P(n) \leq B(x)} \frac{|{n^{-it_1} - n^{-it_2}}|^2}{n} \\
&\ll \sum_{n \leq M(x)} \frac{{(\log n)}^2 |t_1 - t_2|^2}{n} + \sum_{\substack{n > M(x) \\ P(n) \leq B(x)}} \frac{1}{n} \\ 
&\ll |t_1 - t_2|^2 {(\log M(x))}^3 + {M(x)}^{-\sigma} \sum_{P(n) \leq B(x)} \frac{1}{n^{1 - \sigma}} \ll \frac{|t_1 - t_2|^2 {(\log x)}^3}{{(\log \log x)}^{150}} + \frac{1}{{(\log \log x)}^{100}},
\end{align*}
where in the last line we have taken $\sigma = \frac{{(\log \log x)}^{51}}{\log x}$. Note that in the first line the inequality is present because, in the Rademacher case, the Dirichlet series corresponding to the partial Euler product is only over square-free integers. From~\eqref{equ:expofdif} we obtain the upper bound for the probability in~\eqref{equ:spl} of
\begin{multline*} 
\ll {(\log x)}^{3/2} {(\log \log x)}^{-15} {\Biggl( \iint_{\substack{|t_1|, |t_2| \leq {(\log \log x)}^4 \\ |t_1 - t_2| \leq \frac{{(\log \log x)}^5}{\log x}}} \frac{|t_1 - t_2|^2}{|1/2 + it_1|^2} \, d t_1 \, d t_2 \Biggr)}^{1/2} \\ 
+ {(\log \log x)}^{10} {\Biggl( \iint_{\substack{|t_1|, |t_2| \leq {(\log \log x)}^4 \\ |t_1 - t_2| \leq \frac{{(\log \log x)}^5}{\log x}}} \frac{1}{|1/2 + it_1|^2} \, d t_1 \, d t_2 \Biggr)}^{1/2} + \frac{{(\log \log x)}^{50}}{\log x}.
\end{multline*}
Straightforward bounds for these integrals yield the upper bound
\[ 
\ll \frac{1}{{(\log \log x)}^{5}} + \frac{{(\log \log x)}^{15}}{{(\log x)}^{1/2}} + \frac{{(\log \log x)}^{50}}{(\log x)} \ll \frac{1}{{(\log \log x)}^{5}}.
\] 
This completes the proof of Proposition~\ref{p:smallprimes}.
\end{proof}
\subsection{Conditional concentration of the variance on small primes}\label{s:condconc}
Recall from Proposition~\ref{p:smallprimes} that $\mathcal{V}_3 (x)$ is equal to
\[
\frac{{(\log \log x)}^{1/2}}{4 \pi^2} \iint_{\substack{|t_1|, |t_2| \leq {(\log \log x)}^4 \\ |t_1 - t_2| \leq \frac{{(\log \log x)}^5}{\log x}}} \frac{|F_{y} (1/2 + i t_1)|^2 I_y (t_1) \conj{I_y (t_2)}}{|1/2 + it_1|^2} \sum_{\sqrt{x} < p \leq x} \frac{x^{i(t_1 - t_2)}}{p^{1 + i (t_1 - t_2)}} \, d t_1 \, d t_2 ,
\]
where $y = \lfloor 100 \log \log \log x \rfloor$, and 
\[ 
I_y (t) \coloneqq \prod_{\sqrt{x}^{e^{-y}} < p \leq \sqrt{x}} {\Bigl( 1 - \frac{f(p)}{p^{1/2 + i t}} \Bigr)}^{-1}, \quad \text{and} \quad I_y (t) \coloneqq \prod_{\sqrt{x}^{e^{-y}} < p \leq \sqrt{x}} {\Bigl( 1 + \frac{f(p)}{p^{1/2 + i t}} \Bigr)} ,
\]
in the Steinhaus and Rademacher cases, respectively. Let $ \mathcal{V}_4 (x)$ equal
\[ 
\frac{{(\log \log x)}^{1/2}}{4 \pi^2} \iint_{\substack{|t_1|, |t_2| \leq {(\log \log x)}^4 \\ |t_1 - t_2| \leq \frac{{(\log \log x)}^5}{\log x}}} \frac{|F_{y} (1/2 + i t_1)|^2 \E \bigl[ I_y (t_1) \conj{I_y (t_2)} \bigr]}{|1/2 + it_1|^2} \sum_{\sqrt{x} < p \leq x} \frac{x^{i(t_1 - t_2)}}{p^{1 + i (t_1 - t_2)}} \, d t_1 \, d t_2 . 
\] 
In this subsection, we show that $\mathcal{V}_4 (x)$ is a good approximation to $\mathcal{V}_3 (x)$. Furthermore, using Lemma~\ref{l:roughintest}, we show that $\mathcal{V}_4 (x)$ can be approximated by 
\[ 
\mathcal{V}_5 (x) \coloneqq \frac{e^{-\gamma} \log 2}{2\pi} \frac{{(\log \log B(x))}^{1/2}}{\log B(x)} \int_{\R} \frac{|F_{y} (1/2 + it)|^2}{|1/2 + it|^2} \, dt ,
\]
where $B(x) = \sqrt{x}^{e^{-y}}$, thus showing that our variances behave like mean squares of Euler products with high probability. Specifically, we will prove the following two propositions:
\begin{proposition}\label{p:smallprimesconc}
We have 
\[ 
| \mathcal{V}_3 (x) - \mathcal{V}_4 (x) | \ll \frac{1}{{(\log \log x)}^{1/5}} , 
\]
with probability $1 - O \bigl({(\log \log x )}^{-1/5}\bigr)$, in both the Steinhaus and Rademacher cases.
\end{proposition}
\begin{proposition}\label{p:roughestapplication}
We have 
\[ 
| \mathcal{V}_4 (x) - \mathcal{V}_5 (x) | \ll \frac{1}{{(\log \log x)}^{1/5}} , 
\]
with probability $1 - O \bigl({(\log \log x )}^{-1/4}\bigr)$, in both the Steinhaus and Rademacher cases.
\end{proposition}
\begin{remark}
One could obtain a similar result where the range of integration in $\mathcal{V}_5 (x)$ is replaced by $[-W(x), W(x)]$, where $W(x) \rightarrow \infty$.
\end{remark}
\begin{proof}[Proof of Proposition~\ref{p:smallprimesconc}]
This proof is quite involved and constitutes the main innovation of the paper. Beginning with $\mathcal{V}_3 (x)$, which equals
\[
\frac{{(\log \log x)}^{1/2}}{4 \pi^2} \iint_{\substack{|t_1|, |t_2| \leq {(\log \log x)}^4 \\ |t_1 - t_2| \leq \frac{{(\log \log x)}^5}{\log x}}} \frac{|F_{y} (1/2 + i t_1)|^2 I_y (t_1) \conj{I_y (t_2)}}{|(1/2 + it_1)|^2} \sum_{\sqrt{x} < p \leq x} \frac{x^{i(t_1 - t_2)}}{p^{1 + i (t_1 - t_2)}} \, d t_1 \, d t_2 ,
\]
we first remove the piece of the integral where $|t_1| \leq \frac{1}{{(\log \log x)}^{50}}$, allowing us to insert the barrier that appears in Lemma~\ref{l:mcbarrier}. This initial part of the argument is similar to~\cite[Section~3.3]{HarperLargeFluct}. The expected size of the part of the integral where $|t_1| \leq \frac{1}{{(\log \log x)}^{50}}$ is $\ll \frac{1}{{(\log \log x)}^{40}}$. Therefore, by Markov's inequality, $\mathcal{V}_3 (x)$ is equal to
\[
\frac{{(\log \log x)}^{1/2}}{4 \pi^2} \iint_{\substack{|t_1|, |t_2| \leq {(\log \log x)}^4 \\ |t_1 - t_2| \leq \frac{{(\log \log x)}^5}{\log x} \\ |t_1| \geq {(\log \log x)}^{-50}}} \frac{|F_{y} (1/2 + i t_1)|^2 I_y (t_1) \conj{I_y (t_2)}}{|1/2 + it_1|^2} \sum_{\sqrt{x} < p \leq x} \frac{x^{i(t_1 - t_2)}}{p^{1 + i (t_1 - t_2)}} \, d t_1 \, d t_2 ,
\]
up to an additive error of size $\ll \frac{1}{{(\log \log x)}^{10}}$, with probability $1- O({(\log \log x)}^{-10})$.
From Lemma~\ref{l:mcbarrier}, we know that the event $\mathcal{D}(t)$ occurs for all $|t| \leq {(\log \log x)}^{4}$ with probability $1 - O({(\log \log x)}^{-4})$. Therefore, $\mathcal{V}_3 (x)$ equals
\[
\frac{{(\log \log x)}^{1/2}}{4 \pi^2} \iint_{\substack{|t_1|, |t_2| \leq {(\log \log x)}^4 \\ |t_1 - t_2| \leq \frac{{(\log \log x)}^5}{\log x} \\ |t_1| \geq {(\log \log x)}^{-50}}} \frac{\mathbf{1}_{\mathcal{D}(t_1)} |F_{y} (1/2 + i t_1)|^2 I_y (t_1) \conj{I_y (t_2)}}{|1/2 + it_1|^2} \sum_{\sqrt{x} < p \leq x} \frac{x^{i(t_1 - t_2)}}{p^{1 + i (t_1 - t_2)}} \, d t_1 \, d t_2 ,
\]
up to an additive error of size $\ll \frac{1}{{(\log \log x)}^{10}}$, with probability $1- O({(\log \log x)}^{-4})$. Inserting the barrier $\mathcal{A}(t_1)$ from Lemma~\ref{l:mcbarrier}, this is equal to
\begin{multline*}
\frac{{(\log \log x)}^{1/2}}{4 \pi^2} \iint_{\substack{|t_1|, |t_2| \leq {(\log \log x)}^4 \\ |t_1 - t_2| \leq \frac{{(\log \log x)}^5}{\log x} \\ |t_1| \geq {(\log \log x)}^{-50}}} \Biggl( \frac{\mathbf{1}_{\mathcal{D}(t_1)} \mathbf{1}_{\mathcal{A}(t_1)} |F_{y} (1/2 + i t_1)|^2 I_y (t_1) \conj{I_y (t_2)}}{|1/2 + it_1|^2} \\
+ \frac{\mathbf{1}_{\mathcal{D}(t_1)} \mathbf{1}_{\mathcal{A}(t_1) \text{ fails}} |F_{y} (1/2 + i t_1)|^2 I_y (t_1) \conj{I_y (t_2)}}{|1/2 + it_1|^2} \Biggr) \sum_{\sqrt{x} < p \leq x} \frac{x^{i(t_1 - t_2)}}{p^{1 + i (t_1 - t_2)}} \, d t_1 \, d t_2 .
\end{multline*}
By Lemma~\ref{l:mcbarrier} and Lemma~\ref{l:pwprimebound}, and using the fact that $\E |I_y (t_1) \conj{I_y (t_2)}| \ll {(\log \log x)}^{100}$ (which follows from Lemmas~\ref{l:epstein} and~\ref{l:eprad}), the expected size of the contribution from the second term in the parenthesis is
\begin{align*}
\ll & \frac{(\log x) {(\log \log \log x)}^7}{{(\log \log x)}^{1/2}} \iint_{\substack{|t_1| \leq {(\log \log x)}^4 \\ |t_1 - t_2| \leq \frac{{(\log \log x)}^5}{\log x}}} \frac{1}{{(1 + |t_1|)}^2} \min \biggl\{ 1, \frac{1}{|t_1 - t_2| \log x} \biggr\} \, dt_1 \, dt_2 \\
\ll & \frac{{(\log \log \log x)}^7}{{(\log \log x)}^{1/2}} + \frac{{(\log \log \log x)}^7}{{(\log \log x)}^{1/2}} \int_{-{(\log \log x)}^4}^{{(\log \log x)}^4} \frac{d t_1}{{(1+|t_1|)}^2} \int_{\frac{1}{\log x}}^{\frac{{(\log \log x)}^5}{\log x}} \frac{du}{u} \ll \frac{{(\log \log \log x)}^8}{{(\log \log x)}^{1/2}}.
\end{align*}
Hence we find that
\begin{multline*}
\mathcal{V}_3 (x) = \frac{{(\log \log x)}^{1/2}}{4 \pi^2} \iint_{\substack{|t_1|, |t_2| \leq {(\log \log x)}^4 \\ |t_1 - t_2| \leq \frac{{(\log \log x)}^5}{\log x} \\ |t_1| \geq {(\log \log x)}^{-50}}} \frac{\mathbf{1}_{\mathcal{D}(t_1)} \mathbf{1}_{\mathcal{A}(t_1)} |F_{y} (1/2 + i t_1)|^2}{|1/2 + it_1|^2} \\ 
\biggl( I_y (t_1) \conj{I_y (t_2)} \sum_{\sqrt{x} < p \leq x} \frac{x^{i(t_1 - t_2)}}{p^{1 + i (t_1 - t_2)}} \biggr) \, d t_1 \, d t_2 + O \biggl( \frac{1}{{(\log \log x)}^{1/5}} \biggr),
\end{multline*}
with probability $1- O({(\log \log x)}^{-1/5})$. We now consider the quantity
\begin{multline}\label{equ:chebyshev}
\frac{\log \log x}{16 \pi ^4} \E \biggl| \iint_{\substack{|t_1|, |t_2| \leq {(\log \log x)}^4 \\ |t_1 - t_2| \leq \frac{{(\log \log x)}^5}{\log x} \\ |t_1| \geq {(\log \log x)}^{-50}}} \frac{\mathbf{1}_{\mathcal{D}(t_1)} \mathbf{1}_{\mathcal{A}(t_1)} |F_{y} (1/2 + i t_1)|^2}{|1/2 + it_1|^2} \\
\times \Bigr( I_y (t_1) \conj{I_y (t_2)} - \E \Bigl[ I_y (t_1) \conj{I_y (t_2)} \Bigr] \Bigr) \sum_{\sqrt{x} < p \leq x} \frac{x^{i(t_1 - t_2)}}{p^{1 + i (t_1 - t_2)}} \, d t_1 \, d t_2 \biggr|^2 .
\end{multline}
To complete the proof of Proposition~\ref{p:smallprimesconc}, by Chebyshev's inequality, it suffices to show that the above display is $\ll \frac{1}{{(\log \log x)}^{1000}}$, since if this were the case, it would follow that
\begin{multline*}
\mathcal{V}_3 (x) = \frac{{(\log \log x)}^{1/2}}{4 \pi^2} \iint_{\substack{|t_1|, |t_2| \leq {(\log \log x)}^4 \\ |t_1 - t_2| \leq \frac{{(\log \log x)}^5}{\log x} \\ |t_1| \geq {(\log \log x)}^{-50}}} \frac{\mathbf{1}_{\mathcal{D}(t_1)} \mathbf{1}_{\mathcal{A}(t_1)} |F_{y} (1/2 + i t_1)|^2}{|1/2 + it_1|^2} \\ 
\biggl( \E \Bigl[ I_y (t_1) \conj{I_y (t_2)} \Bigr] \sum_{\sqrt{x} < p \leq x} \frac{x^{i(t_1 - t_2)}}{p^{1 + i (t_1 - t_2)}} \biggr) \, d t_1 \, d t_2 + O \biggl( \frac{1}{{(\log \log x)}^{1/5}} \biggr),
\end{multline*}
with probability $1- O({(\log \log x)}^{-1/5})$. The barriers can then be removed by reversing the argument used to introduce them, and this delivers $\mathcal{V}_4 (x)$. \\
Expanding out the square, we find that equation~\eqref{equ:chebyshev} is
\begin{multline}\label{equ:varnotsplitup}
\ll \log \log x \iiiint_{\substack{|t_1|, |t_2|, |t_3|, |t_4| \leq {(\log \log x)}^4 \\ |t_1|, |t_3| \geq {(\log \log x)}^{-50} \\ |t_1 - t_2|, |t_3 - t_4| \leq \frac{{(\log \log x)}^5}{\log x} }} \E \mathbf{1}_{\mathcal{A}(t_1)} | F_{y} (1/2 + it_1) |^2 \mathbf{1}_{\mathcal{A}(t_3)} | F_{y} (1/2 + it_3) |^2 \\ 
\times \bigl| \E I_y (t_1) \conj{I_y (t_2)} \conj{I_y (t_3)} I_y (t_4) - \E I_y(t_1) \conj{I_y (t_2)} \E \conj{I_y (t_3)} I_y (t_4) \bigr| \, dt_1 \, dt_2 \, dt_3 \, dt_4.
\end{multline}
When $t_1 \approx t_3$ (and additionally, when $t_1 \approx - t_3$ in the Rademacher case), the Euler products can reinforce one another and give a large contribution. It is the barriers $\mathcal{A}(t_1)$ and $\mathcal{A}(t_3)$ that will allow us to deal with the contribution from such points. On the remaining points, we will be able to apply Lemma~\ref{l:largeprimeav} to give a saving. \\ 
With this in mind, we first consider the portions of the integral in~\eqref{equ:varnotsplitup} where $|t_1 - t_3| \leq {(\log x)}^{-1/4}$ or $ | t_1 + t_3 | \leq {(\log x)}^{-1/4}$. On these points, we apply the trivial upper bound
\[
\bigl| \E I_y (t_1) \conj{I_y (t_2)} \conj{I_y (t_3)} I_y (t_4) - \E I_y (t_1) \conj{I_y (t_2)} \E \conj{I_y (t_3)} I_y (t_4) \bigr| \ll {(\log \log x)}^{400} ,
\]
for any $t_1, t_2, t_3, t_4 \in \R$, which follows from applying the triangle inequality and Cauchy--Schwarz. Then integrating over $t_2$ and $t_4$, we obtain the bound
\begin{multline*} 
\ll \frac{{(\log \log x)}^{411}}{{(\log x)}^2} \Biggl[ \iint_{\substack{|t_1|, |t_3| \leq {(\log \log x)}^4 \\ |t_1|, |t_3| \geq {(\log \log x)}^{-50} \\ |t_1 - t_3| \leq {(\log x)}^{-1/4}}} \E \mathbf{1}_{\mathcal{A}(t_1)} | F_{y} (1/2 + it_1) |^2 \mathbf{1}_{\mathcal{A}(t_3)} | F_{y} (1/2 + it_3) |^2 \, dt_1 \, dt_3 \\
+ \iint_{\substack{|t_1|, |t_3| \leq {(\log \log x)}^4 \\ |t_1|, |t_3| \geq {(\log \log x)}^{-50} \\ |t_1 + t_3| \leq {(\log x)}^{-1/4}}} \E \mathbf{1}_{\mathcal{A}(t_1)} | F_{y} (1/2 + it_1) |^2 \mathbf{1}_{\mathcal{A}(t_3)} | F_{y} (1/2 + it_3) |^2 \, dt_1 \, dt_3 \Biggr] .
\end{multline*}
Both of these parts will be handled identically, so we restrict ourselves to the first term. Define 
\[ 
J(t) \coloneqq \lfloor \max \{ {100 \log \log \log x}, \log ( |t| \log x ) \} \rfloor .
\]
In the integral
\[ 
\frac{{(\log \log x)}^{411}}{{(\log x)}^2} \iint_{\substack{|t_1|, |t_3| \leq {(\log \log x)}^4 \\ |t_1|, |t_3| \geq {(\log \log x)}^{-50} \\ |t_1 - t_3| \leq {(\log x)}^{-1/4}}} \E \mathbf{1}_{\mathcal{A}(t_1)} | F_{y} (1/2 + it_1) |^2 \mathbf{1}_{\mathcal{A}(t_3)} | F_{y} (1/2 + it_3) |^2 \, d t_1 \, d t_3 ,
\]
the Euler products reinforce one another on primes $p$ where $p \leq e^{1/|t_1 - t_3|}$. Therefore, we decompose one of the Euler products, writing the above as
\begin{multline*}
\frac{{(\log \log x)}^{411}}{{(\log x)}^2} \iint_{\substack{|t_1|, |t_3| \leq {(\log \log x)}^4 \\ |t_1|, |t_3| \geq {(\log \log x)}^{-50} \\ |t_1 - t_3| \leq {(\log x)}^{-1/4}}} \E \biggl[ \mathbf{1}_{\mathcal{A}(t_1)} |F_{J(t_1 - t_3)} (1/2 + i t_1)|^2 \\
\times \biggl| \frac{F_y (1/2 + it_1)}{F_{J(t_1 - t_3)} (1/2 + it_1)} \biggr|^2 \mathbf{1}_{\mathcal{A}(t_3)} | F_{y} (1/2 + it_3) |^2 \biggr] \, d t_1 \, d t_3 ,
\end{multline*}
From the definition of $\mathcal{A}(t_1)$, we have 
\[ 
\mathbf{1}_{\mathcal{A}(t_1)} |F_{J(t_1 - t_3)} (1/2 + i t_1)|^2 \ll \frac{1}{{(\log \log x)}^{2000}}{\biggl( \min \biggl\{ \frac{1}{|t_1 - t_3|}, \frac{\log x}{{(\log \log x)}^{100}} \biggr\} \biggr)}^2 .
\] 
Applying this bound and removing the indicator function $\mathbf{1}_{\mathcal{A}(t_3)}$ gives the bound
\begin{multline}\label{equ:varconstructive}
\ll \frac{1}{{(\log x)}^2 {(\log \log x)}^{1500}} \iint_{\substack{|t_1|, |t_3| \leq {(\log \log x)}^4 \\ |t_1|, |t_3| \geq {(\log \log x)}^{-50} \\ |t_1 - t_3| \leq {(\log x)}^{-1/4}}} {\biggl( \min \biggl\{ \frac{1}{|t_1 - t_3|}, \frac{\log x}{{(\log \log x)}^{100}} \biggr\} \biggr)}^2 \\ \times \E | F_{J(t_1 - t_3)} (1/2 + it_3) |^2 \E \biggl| \frac{F_y (1/2 + it_1)}{F_{J(t_1 - t_3)} (1/2 + it_1)} \biggr|^2 \biggl| \frac{F_y (1/2 + it_3)}{F_{J(t_1 - t_3)} (1/2 + it_3)} \biggr|^2 \, d t_1 \, d t_3 .
\end{multline}
Note that the products $|F_{J(t_1 - t_3)} (1/2 + i t)|$ and $\bigl| \frac{F_y (1/2 + it)}{F_{J(t_1 - t_3)} (1/2 + it)} \bigr|$ don't share any primes, so we have used independence to factor the expectation. By Lemmas~\ref{l:epstein} and~\ref{l:eprad}, we have $\E | F_{J(t_1 - t_3)} (1/2 + it_3) |^2 \ll \min \bigl\{ \frac{1}{|t_1 - t_3|}, \frac{\log x}{{(\log \log x)}^{100}} \bigr\}$ in both the Steinhaus and Rademacher cases. Furthermore, by the second parts of Lemmas~\ref{l:epstein} and~\ref{l:eprad}, we have
\begin{multline*}
\E \biggl| \frac{F_y (1/2 + it_1)}{F_{J(t_1 - t_3)} (1/2 + it_1)} \biggr|^2 \biggl| \frac{F_y (1/2 + it_3)}{F_{J(t_1 - t_3)} (1/2 + it_3)} \biggr|^2 \ll {\biggl( \frac{e^{J(t_1 - t_3)}}{{(\log \log x)}^{100}} \biggr)}^2 \\ 
\times {\biggl( 1 + \min \biggl\{ \frac{e^{J(t_1 - t_3)}}{{(\log \log x)}^{100}}, \frac{e^{J(t_1 - t_3)}}{|t_1 + t_3| \log x} \biggr\} \biggr)}^2 {\biggl( 1 + \min \biggl\{ \frac{e^{J(t_1 - t_3)}}{{(\log \log x)}^{100}}, \frac{e^{J(t_1 - t_3)}}{|t_1 - t_3|\log x} \biggr\} \biggr)}^2 ,
\end{multline*}
on our range of integration, seeing as the first terms in these results are always bounded on this range. Note that when $|t_1 - t_3| \leq \frac{{(\log \log x)}^{100}}{\log x}$, the expectation above is $\ll 1$, seeing as $J(t_1 - t_3) = y$ in this case. The contribution to~\eqref{equ:varconstructive} from this case is bounded above by
\[ 
\ll \frac{1}{{(\log x)}^2 {(\log \log x)}^{1500}} \iint_{\substack{|t_1|, |t_3| \leq {(\log \log x)}^4 \\ |t_1 - t_3| \leq \frac{{(\log \log x)}^{100}}{\log x}}} \frac{{(\log x)}^{3}}{{(\log \log x)}^{300}} \, dt_1 \, dt_3 \ll \frac{1}{{(\log \log x)}^{1000}},
\]
say. Furthermore, when $\frac{{(\log \log x)}^{100}}{\log x} < |t_1 - t_3| \leq \frac{1}{{(\log x)}^{1/4}}$, we have the bound
\begin{multline*}
\E \biggl| \frac{F_y (1/2 + it_1)}{F_{J(t_1 - t_3)} (1/2 + it_1)} \biggr|^2 \biggl| \frac{F_y (1/2 + it_3)}{F_{J(t_1 - t_3)} (1/2 + it_3)} \biggr|^2 \\ 
\ll {\biggl( \frac{|t_1 - t_3| \log x}{{(\log \log x)}^{100}} \biggr)}^2 {\biggl( 1 + \min \biggl\{ \frac{|t_1 - t_3| \log x}{{(\log \log x)}^{100}}, \frac{|t_1 - t_3|}{|t_1 + t_3|} \biggr\} \biggr)}^2 .
\end{multline*}
Hence the contribution to~\eqref{equ:varconstructive} from the part of the integral where $|t_1 - t_3| > \frac{{(\log \log x)}^{100}}{\log x}$ is
\[
\ll \frac{1}{{(\log \log x)}^{1700}} \iint_{\substack{|t_1|, |t_3| \leq {(\log \log x)}^4 \\ |t_1|, |t_3| \geq {(\log \log x)}^{-50} \\ |t_1 - t_3| \leq {(\log x)}^{-1/4} \\|t_1 - t_3| > \frac{{(\log \log x)}^{100}}{\log x}}} \frac{1}{|t_1 - t_3|} {\biggl( 1 + \min \biggl\{ \frac{|t_1 - t_3| \log x}{{(\log \log x)}^{100}}, \frac{|t_1 - t_3|}{|t_1 + t_3|} \biggr\} \biggr)}^2 \, dt_1 \, dt_3 .
\]
By the reverse triangle inequality, one can deduce that $\frac{1}{{(\log \log x)}^{50}} \leq |t_1 + t_3| $ for any $(t_1, t_3)$ in the domain of integration, when $x$ is large. Therefore, the term in the parenthesis is $\ll 1$, and we can crudely bound this by
\[
\ll \frac{1}{{(\log \log x)}^{1700}} \iint_{\substack{|t_1|, |t_3| \leq {(\log \log x)}^4 \\ |t_1|, |t_3| \geq {(\log \log x)}^{-50} \\ |t_1 - t_3| \leq {(\log x)}^{-1/4} \\|t_1 - t_3| > \frac{{(\log \log x)}^{100}}{\log x}}} \frac{1}{|t_1 - t_3|} \, dt_1 \, dt_3 \ll \frac{1}{{(\log \log x)}^{1600}} ,
\]
say. Thus, we have found that the contribution to~\eqref{equ:varnotsplitup} from points that satisfy $|t_1 - t_3| \leq {(\log x)}^{-1/4}$ (and similarly, $|t_1 + t_3| \leq {(\log x)}^{-1/4}$) is $\ll \frac{1}{{(\log \log x)}^{1000}}$, which suffices for our purpose. It remains to show that the contribution to~\eqref{equ:varnotsplitup} from points where $|t_1 - t_3|$, $|t_1 + t_3| \geq {(\log x)}^{-1/4}$ satisfies the same bound. Dropping the barriers $\mathbf{1}_{A(t_1)}$ and $\mathbf{1}_{A(t_3)}$ and the lower bounds for $|t_1|$ and $|t_3|$, we get a contribution of
\begin{multline}\label{equ:nonconstr}
\ll \log \log x \iiiint_{\substack{|t_1|, |t_3| \leq {(\log \log x)}^4 \\ |t_1 - t_2|, |t_3 - t_4| \leq \frac{{(\log \log x)}^5}{\log x} \\ |t_1 - t_3|, |t_1 + t_3| \geq {(\log x)}^{-1/4} }} \E | F_{y} (1/2 + it_1) |^2 | F_{y} (1/2 + it_3) |^2 \\ 
\times \bigl| \E I_y (t_1) \conj{I_y (t_2)} \conj{I_y (t_3)} I_y (t_4) - \E I_y (t_1) \conj{I_y (t_2)} \E \conj{I_y (t_3)} I_y (t_4) \bigr| \, dt_1 \, dt_2 \, dt_3 \, dt_4, 
\end{multline}
By the second parts of Lemmas~\ref{l:epstein} and~\ref{l:eprad}, the expectation $\E | F_{y} (1/2 + it_1) |^2 | F_{y} (1/2 + it_3) |^2$ is 
\begin{align*}
\ll & {(\log \log x)}^{4/25} {\biggl( \frac{\log x}{{(\log \log x)}^{100}} \biggr)}^2 {\biggl( 1 + \min \biggl\{ \frac{\log x}{{(\log \log x)}^{100}} , \frac{1}{|t_1 + t_3|} \biggr\} \biggr)}^2 \\ 
& \times {\biggl( 1 + \min \biggl\{ \frac{\log x}{{(\log \log x)}^{100}} , \frac{1}{|t_1 - t_3|} \biggr\} \biggr)}^2 \\
\ll & \frac{{(\log x)}^2}{{(\log \log x)}^{190}} {\biggl( 1 + \frac{1}{|t_1 + t_3|} \biggr)}^2 {\biggl( 1 + \frac{1}{|t_1 - t_3|}\biggr)}^2 ,
\end{align*}
on the range of integration. Inserting this and applying Lemma~\ref{l:largeprimeav} to handle the latter expectation, we find that~\eqref{equ:nonconstr} is
\begin{multline*}
\ll (\log x){(\log \log x)}^{150} \iiiint_{\substack{|t_1|, |t_3| \leq {(\log \log x)}^4 \\ |t_1 - t_2|, |t_3 - t_4| \leq \frac{{(\log \log x)}^5}{\log x} \\ |t_1 - t_3|, |t_1 + t_3| \geq {(\log x)}^{-1/4} }} {\biggl( 1 + \frac{1}{|t_1 + t_3|} \biggr)}^2 {\biggl( 1 + \frac{1}{|t_1 - t_3|} \biggr)}^2 \\ 
\biggl( \frac{1}{|t_1 - t_3|} + \frac{1}{|t_2 - t_4|} + \frac{1}{|t_1 + t_4|} + \frac{1}{|t_2 + t_3|} \biggr) \, dt_1 \, dt_2 \, dt_3 \, dt_4 ,
\end{multline*}
say. Furthermore, $|t_1 + t_4| \asymp |t_2 + t_3| \asymp |t_1 + t_3|$ and $|t_2 - t_4| \asymp |t_1 - t_3|$, allowing us to completely remove the dependence on $t_2$ and $t_4$. We therefore obtain a bound for the previous display of
\begin{multline}\label{equ:t1t3int}
\ll \frac{{(\log \log x)}^{160}}{\log x} \iint_{\substack{|t_1|, |t_3| \leq {(\log \log x)}^4 \\ |t_1 - t_3|, |t_1 + t_3| \geq {(\log x)}^{-1/4} }} {\biggl( 1 + \frac{1}{|t_1 + t_3|} \biggr)}^2 {\biggl( 1 + \frac{1}{|t_1 - t_3|} \biggr)}^2 \\ 
\biggl( \frac{1}{|t_1 - t_3|} + \frac{1}{|t_1 + t_3|} \biggr) \, dt_1 \, dt_3 ,
\end{multline}
The contribution to~\eqref{equ:t1t3int} from the portion of the integral where $|t_1 + t_3 | \leq 1$ is
\begin{multline*}
\ll \frac{{(\log \log x)}^{160}}{\log x} \iint_{\substack{|t_1|, |t_3| \leq {(\log \log x)}^4 \\ |t_1 - t_3|, |t_1 + t_3| \geq {(\log x)}^{-1/4} \\ |t_1 + t_3| \leq 1}} {\biggl( 1 + \frac{1}{|t_1 - t_3|} \biggr)}^2 \\ 
\times \biggl( \frac{1}{|t_1 - t_3||t_1 + t_3|^2} + \frac{1}{|t_1 + t_3|^3} \biggr) \, dt_1 \, dt_3 .
\end{multline*}
The contribution to this integral from $|t_1 - t_3| > 1$ is
\[ 
\ll \frac{{(\log \log x)}^{160}}{\log x} \iint_{\substack{|t_1|, |t_3| \leq {(\log \log x)}^4 \\ |t_1 + t_3| \geq {(\log x)}^{-1/4}}} \frac{1}{|t_1 + t_3|^3} \, dt_1 \, dt_3 \ll \frac{{(\log \log x)}^{170}}{{(\log x)}^{1/2}} ,
\]
and from $|t_1 - t_3| \leq 1$ the contribution is
\[ 
\ll \frac{{(\log \log x)}^{160}}{\log x} \iint_{\substack{|t_1|, |t_3| \leq {(\log \log x)}^4 \\ {(\log x)}^{-1/4} \leq |t_1 - t_3|, |t_1 + t_3| \leq 1}} \biggl( \frac{1}{|t_1 - t_3|^3 |t_1 + t_3|^2} + \frac{1}{|t_1 - t_3|^2 |t_1 + t_3|^3} \biggr) \, dt_1 \, dt_3 .
\]
By symmetry, it suffices to bound the contribution from only one of the terms in the parenthesis, and it is straightforward to obtain the upper bound
\[
\ll \frac{{(\log \log x)}^{160}}{\log x} \int_{u \geq {(\log x)}^{-1/4}} \frac{du}{u^3} \int_{v \geq {(\log x)}^{-1/4}} \frac{dv}{v^2} \ll \frac{{(\log \log x)}^{160}}{{(\log x)}^{1/4}} .
\]
Finally, the contribution to~\eqref{equ:t1t3int} from the portion of the integral where $|t_1 + t_3 | > 1$ is
\[ 
\ll \frac{{(\log \log x)}^{160}}{\log x} \iint_{\substack{|t_1|, |t_3| \leq {(\log \log x)}^4 \\ |t_1 - t_3| \geq {(\log x)}^{-1/4} \\ |t_1 + t_3 | >1}} {\biggl( 1 + \frac{1}{|t_1 - t_3|} \biggr)}^3 \, dt_1 \, dt_3 ,
\]
similar arguments show that this is $\ll \frac{{(\log \log x)}^{170}}{{(\log x)}^{1/4}}$. Putting this all together, we have shown that~\eqref{equ:chebyshev} is $\ll \frac{1}{{(\log \log x)}^{1000}}$, completing the proof of Proposition~\ref{p:smallprimesconc}.
\end{proof}
We proceed with the proof of Proposition~\ref{p:roughestapplication}, which is an application of Lemma~\ref{l:roughintest}.
\begin{proof}[Proof of Proposition~\ref{p:roughestapplication}]
We begin with $\mathcal{V}_4 (x)$, which is
\[ 
\frac{{(\log \log x)}^{1/2}}{4 \pi^2} \iint_{\substack{|t_1|, |t_2| \leq {(\log \log x)}^4 \\ |t_1 - t_2| \leq \frac{{(\log \log x)}^5}{\log x}}} \frac{|F_{y} (1/2 + i t_1)|^2 \E \bigl[ I_y (t_1) \conj{I_y (t_2)} \bigr]}{|1/2 + it_1|^2} \sum_{\sqrt{x} < p \leq x} \frac{x^{i(t_1 - t_2)}}{p^{1 + i (t_1 - t_2)}} \, d t_1 \, d t_2 . 
\]
Note that we can drop the condition $|t_2| \leq {(\log \log x)}^4 $ at a small cost, seeing as the conditions that $|t_1| \leq {(\log \log x)}^4 $ and $|t_1 - t_2 | \leq \frac{{(\log \log x)}^5}{\log x}$ almost enforce this condition. Specifically, an application of Markov's inequality allows one to find that dropping this condition causes an additive error of size $\ll \frac{1}{{(\log x)}^{1/4}}$ with probability $1 - O \bigl( {(\log x)}^{-1/4} \bigr)$, say. Furthermore, as noted in the proof of Lemma~\ref{l:largeprimeav}, we have
\[ 
\E \Bigl[ I_y (t_1) \conj{I_y (t_2)} \Bigr] = \prod_{\sqrt{x}^{e^{-y}} < p \leq \sqrt{x}} \biggl( 1 + \frac{1}{p^{1 + i (t_1 - t_2)}} + O \biggl( \frac{1}{p^{3/2}} \biggr) \biggr),
\]
in both the Steinhaus and Rademacher cases. Recall that $y = \lfloor 100 \log \log \log x \rfloor$ and $B(x) = \sqrt{x}^{e^{-y}}$. Performing the change of variables $t_2 = t_1 - u$ allows one to find that $\mathcal{V}_4 (x)$ is equal to 
\[ 
\frac{{(\log \log x)}^{1/2}}{4 \pi^2} \int_{|t_1|\leq {(\log \log x)}^4} \biggl| \frac{F_{y} (1/2 + i t_1)}{1/2 + it_1} \biggr|^2 \int_{|u| \leq \frac{{(\log \log x)}^5}{\log x}} g(u) \, du \, d t_1, 
\]
where
\[ 
g(u) = \prod_{B(x) < p \leq \sqrt{x}} \biggl( 1 + \frac{1}{p^{1 + i u}} + O \biggl( \frac{1}{p^{3/2}} \biggr) \biggr) \sum_{\sqrt{x} < p \leq x} \frac{x^{iu}}{p^{1 + iu}},
\]
up to an error of size $\ll \frac{1}{{(\log x)}^{1/4}}$, with probability $1 - O ( {(\log x)}^{-1/4} )$. By Lemma~\ref{l:roughintest}, we have
\[ 
\int_{|u| \leq \frac{{(\log \log x)}^5}{\log x}} g(u) \, du = \frac{2 \pi e^{-\gamma} \log 2}{\log B(x)} \biggl( 1 + O \biggl( \frac{1}{\log \log x} \biggr) \biggr) .
\]
Thus, $\mathcal{V}_4 (x)$ is equal to 
\[ 
\frac{e^{-\gamma} \log 2 {(\log \log x)}^{1/2}}{2\pi \log B(x)} \biggl( 1 + O \biggl( \frac{1}{\log \log x} \biggr) \biggr) \int_{|t| \leq {(\log \log x)}^4} \biggl| \frac{F_{y} (1/2 + i t)}{1/2 + it} \biggr|^2 \, dt .
\]
Since ${(\log \log x)}^{1/2} = {(\log \log B(x))}^{1/2} + O ( \frac{\log \log \log x}{{(\log \log x)}^{1/2}} )$ we find that $\mathcal{V}_4 (x)$ is equal to
\[ 
\frac{e^{-\gamma} \log 2 {(\log \log B(x))}^{1/2}}{2\pi \log B(x)} \biggl( 1 + O \biggl( \frac{\log \log \log x}{{(\log \log x)}^{1/2}} \biggr) \biggr) \int_{|t| \leq {(\log \log x)}^4} \biggl| \frac{F_{y} (1/2 + i t)}{1/2 + it} \biggr|^2 \, dt .
\]
It follows from Markov's inequality and Lemma~\ref{l:mcexpectation} that
\[ 
\int_{\R} \biggl| \frac{F_{y} (1/2 + i t)}{1/2 + it} \biggr|^2 \, dt \leq \frac{\log B(x)}{{(\log \log x)}^{9/40}},
\]
with probability $1 - O ( {(\log \log x)}^{-1/4} )$, say. Hence $\mathcal{V}_4 (x)$ is equal to
\[ 
\frac{e^{-\gamma} \log 2 {(\log \log B(x))}^{1/2}}{2\pi \log B(x)} \int_{|t| \leq {(\log \log x)}^4} \biggl| \frac{F_{y} (1/2 + i t)}{1/2 + it} \biggr|^2 \, dt 
\]
plus an error of size at most $O ( {(\log \log x)}^{-1/5} )$, with probability $ 1 - O ( {(\log \log x)}^{-1/4} )$. Another application of Markov's inequality and Lemmas~\ref{l:epstein} and~\ref{l:eprad} similarly allow us to extend the integral to the whole real line. This completes the proof of Proposition~\ref{p:roughestapplication}.
\end{proof}
Theorem~\ref{t:variance} is completed by combining Propositions~\ref{p:perron},~\ref{p:smallprimes},~\ref{p:smallprimesconc}, and~\ref{p:roughestapplication}.

\section{Distribution of the variance: proof of Theorem~\ref{t:SWext}}\label{s:distribofvar} 
In this section, we prove Theorem~\ref{t:SWext}. Throughout this section, we will use the notation $F_{(x)} = \prod_{p \leq x} {\bigl( 1 - \frac{f(p)}{p^{1/2 + it}} \bigr)}^{-1}$. Theorem~\ref{t:SWext} will be a consequence of~\citet[Theorem~1.9]{SaksmanWebb}, which we quote below with our notation. 
\begin{proposition}[Theorem~1.9 of~\citet{SaksmanWebb}]
Let $\pi (x) = \# \{ p \leq x \}$. Fix any nonnegative continuous function $h \colon [0,1] \rightarrow [0,\infty)$. As $x \rightarrow \infty$, we have
\[
{\bigl(\log \log \pi(x)\bigr)}^{1/2} \int_{0}^{1} h(t) \frac{|F_{(x)} (1/2 + it)|^2}{\E |F_{(x)} (1/2 + it)|^2} \, dt \xrightarrow{d} \int_{0}^{1} h(t) g(t) \lambda (dt),
\]
where $g$ is a positive random continuous function such that the norms $\| g \|_{C^l [0,1]}$ and $\| 1 / g \|_{C^l [0,1]}$ possess moments of all orders, and $\lambda (dt)$ is a critical Gaussian multiplicative chaos measure.
\end{proposition}
The function $\pi (x)$ appears because their Euler product is taken over the first $x$ primes. It can easily be replaced by $x$ seeing as the iterated logarithm grows slowly.
\subsection{The work of Saksman and Webb and the Rademacher case}\label{s:SW Outline}
Before proving Theorem~\ref{t:SWext}, we begin with a brief overview of the proof of~\cite[Theorem~1.9]{SaksmanWebb}, followed by the difficulties we encounter when modifying the proof to the Rademacher case. We adapt the notation so that it agrees with Theorem~\ref{t:SWext}, and sketch their proof over the compact domain $[-n,n]$. Beginning with
\[
\frac{{(\log \log x)}^{1/2}}{\log x} \int_{-n}^n \frac{e^{2 \Re \log F_{(x)} (1/2 + it)}}{|1/2 + it|^2} \, dt,
\]
the goal is to approximate $\Re \log F_{(x)} (1/2 + it)$ in the exponent by a Gaussian field for which convergence of the integral is known. If the Gaussian field approximates $2 \Re \log F_{(x)} (1/2 + it)$ well, then the (random) error of the approximation should possess nice properties, and will not affect the distributional convergence. This error is captured by the function $g$ in our results.~\citet{SaksmanWebb} perform many approximation steps to obtain a Gaussian field for which convergence is known. The first step is to apply Taylor expansion and obtain
\begin{equation}\label{equ:swfirststep}
\exp (2 \Re \log F_{(x)} (1/2 + it)) = \exp \biggl( 2 \Re \sum_{p \leq x} \frac{f(p)}{p^{1/2 + it}} + \text{ nice random error term} \biggr),
\end{equation}
where the error contributes to the function $g$. They then replace the $f(p)$'s by complex Gaussian random variables, $W_p$ (defined on the same probability space). The error induced is non-trivial seeing as the individual $f(p)$ \emph{are not} Gaussian, but they show that the error remains under control (see~\cite[Theorem~1.7]{SaksmanWebb}). They then work to replace the Gaussian random variables by integrals with respect to Brownian motion, and ultimately end up with an approximation to the original field $2 \Re \log F_{(x)} (1/2 + it)$ which is approximately
\[ 
G_{x} (t) = 2 \Re \int_{1}^{\log x} \sqrt{2 \widehat{C} (s)} e^{-2\pi i t s} \, d B_s^{\C}, 
\]
where $C(x) = \max\{ - \log |x|, 0 \}$, $\widehat{C}$ denotes its Fourier transform, and $d B_s^{\C}$ is a standard complex Brownian motion defined on the same probability space as ${(f(p))}_{p \text{ prime}}$. Note that the function $(\widehat{C} (s))^{1/2}$ directly corresponds to $\frac{1}{p^{1/2}}$ in~\eqref{equ:swfirststep}, $e^{-2\pi i t s}$ corresponds to $p^{-it}$, and the complex Brownian motion to our random variables $f(p)$. We can then write our original integral as
\begin{equation}\label{equ:swfinalint}
\frac{{(\log \log x)}^{1/2}}{\log x} \int_{-n}^n \frac{e^{2 \Re \log F_{(x)} (1/2 + it) - G_x (t)}}{|1/2 + it|^2} e^{G_x (t)} \, dt.
\end{equation}
One can then think of $g$ as the function $g(t) = \lim_{x \rightarrow \infty} e^{2 \Re \log F_{(x)} (1/2 + it) - G_x (t)}$ which is almost surely defined pointwise. Such a definition only makes sense if one has this underlying probability space, so we cannot give an explicit definition in Theorems~\ref{t:main},~\ref{t:SWext}, and Proposition~\ref{p:epconvonboundeddomains}. The proof is completed by noting that the covariances of $G_x$ are $\E G_x (t_1) G_x (t_2) = \max \{- 2 \log |t_1 - t_2|, 0 \} + o(1)$ for any fixed $t_1$ and $t_2$, and convergence of~\eqref{equ:swfinalint} follows by a covariance comparison result of~\citet[Theorem~1]{JunSaks}, utilising the fact that convergence is known for other Gaussian fields with the same covariance structure. \\

The Rademacher case presents some difficulties that we do not overcome here. By symmetry of the integral, we need to understand
\begin{equation}\label{equ:radintegral}
\frac{{(\log \log x)}^{1/2}}{\log x} \int_{0}^{n} \frac{e^{2 \Re \sum_{p \leq x} \frac{f(p)}{p^{1/2 + it}} + E_x (t)}}{| \zeta (1 + 2it) | |1/2 + it|^2} \, dt ,
\end{equation}
where $E_x(t)$ is a nice random error function and ${(f(p))}_{p \text{ prime}}$ are independent Rademacher random variables. The appearance of the zeta function comes from higher order terms in the Taylor expansion of $\Re \log F(1/2 + it)$ (for example, the second order term is $- \frac{1}{2} \sum_{p \leq x} \frac{1}{p^{1 + 2it}}$), noting that the Euler product formula for $\zeta (1 + 2 i t)$ converges conditionally when $t \neq 0$ (see~\cite[Section 3.15.1]{Titchmarsh}). One can try and proceed as in the Steinhaus case, beginning with $2 \Re \bigl( \sum_{p \leq x} \frac{f(p)}{p^{1/2 + it}} \bigr)$ and performing a Gaussian field approximation. By identical techniques to the Steinhaus case, one can obtain the analogous approximation
\[ 
G^{\mathrm{(Rad)}}_{x} = 2 \Re \int_{0}^{\log x} \sqrt{2 \widehat{C} (s)} e^{-2 \pi i t s} \, d B_s ,
\]
where $B_s$ is a standard real Brownian motion. Note that since the Rademacher random multiplicative function takes only real values, the Brownian motion does too. This is different to the Steinhaus case where we have a standard complex Brownian motion, and a consequence is that we have covariances $\E G^{\mathrm{(Rad)}}_{x} (t_1) G^{\mathrm{(Rad)}}_{x} (t_2) = \max \{- 2 \log |t_1 - t_2|, 0 \} + \max \{- 2 \log |t_1 + t_2|, 0 \} + o(1)$. It is unclear how to obtain convergence of the integrals when the covariances have this additional factor involving $\log |t_1 + t_2|$.

\subsection{Proof of Theorem~\ref{t:SWext}}\label{s:completionoffullintconv}
We begin the proof with a slight generalisation of~\citet[Theorem~1.9]{SaksmanWebb}.
\begin{proposition}[Adaptation of Theorem 1.9 of~\citet{SaksmanWebb}]\label{p:epconvonboundeddomains}
Fix $n \in \N$, and take any nonnegative continuous function $h \colon [-n,n] \rightarrow \R$. For $f$ a Steinhaus random multiplicative function, as $x \rightarrow \infty$, we have
\[ 
\frac{{(\log \log x)}^{1/2}}{\log x} \int_{-n}^{n} h(t) \biggl| \frac{F_{(x)} (1/2 + it)}{1/2 + it} \biggr|^2 \, dt \xrightarrow{d} \int_{-n}^{n} \frac{h(t) g(t)}{|1/2 + it|^2} \lambda (dt) ,
\]
where $\lambda (dt)$ is a critical multiplicative chaos measure and $g(t)$ is a positive random continuous function such that $\| g \|_{C^l [-n,n]}$ and $\| 1 / g \|_{C^l [-n,n]}$ possess moments of all orders. The random function $g (t)$ and the random measure $\lambda (dt)$ are independent of $n$.
\end{proposition}
\begin{remark}
Proposition~\ref{p:epconvonboundeddomains} tells us that for each $n$, $\| g \|_{C^l [-n,n]}$ and $\| 1/g \|_{C^l [-n,n]}$ possess moments of all orders, but does not give dependence on $n$. It can further be shown that, for any $\varepsilon > 0$, $\| \frac{g(t)}{|1/2 + it|^\varepsilon} \|_{L^\infty (\R)}$ possesses moments of all orders, but the calculations are omitted for brevity.
\end{remark}
\begin{proof}[Proof of Proposition~\ref{p:epconvonboundeddomains}]
We will describe the necessary adaptations to~\cite[Theorem~1.9]{SaksmanWebb}. First of all, their result gives convergence of integrals over $[0,1]$, but their proof readily adapts to arbitrary compact intervals, say $[-n,n]$ for $n \in \N$. Therefore, as $x \rightarrow \infty$, we have
\[ 
{\bigl(\log \log \pi (x) \bigr)}^{1/2} \int_{-n}^{n} h(t) \frac{|F_{(x)} (1/2 + it)|^2}{{\E |F_{(x)} (1/2 + it)|^2}} \, dt \xrightarrow{d} \int_{-n}^{n} h(t) g(t) \lambda (dt) ,
\]
for any nonnegative continuous function $h \colon [-n,n] \rightarrow \R$, where $g$ has the properties stated in the proposition. Looking at the proof of~\cite[Theorem~1.9]{SaksmanWebb}, we see that $g$ and $\lambda$ must be independent of $n$, seeing as increasing $n$ just extends them to a larger domain. We wish to replace $\E |F_{x} (1/2 + it)|^2$ in the denominator by $\log x$. It follows from Euler Product Result 1 of~\cite{HarperHM} and Mertens' second theorem that $\E |F_{(x)} (1/2 + it)|^2 = K \log x (1 + O(1/\log x))$, for some fixed constant $K > 0$. The factor of $K$ can be engulfed by the function $g(t)$ in Proposition~\ref{p:epconvonboundeddomains}. We have
\[ 
\frac{{\bigl( \log \log \pi (x) \bigr)}^{1/2}}{\log x (1 + O(1/ \log x))} = \frac{{(\log \log x)}^{1/2}}{\log x} + O \biggl( \frac{\log \log x}{{(\log x)}^2} \biggr),
\]
and an application of Markov's inequality shows that the contribution from the error term converges in probability to zero, so
has no impact on convergence in distribution. Relabelling $h \rightarrow \frac{1}{|1/2 + it|^2} h$ gives Proposition~\ref{p:epconvonboundeddomains}.
\end{proof}
\begin{proof}[Proof of Theorem~\ref{t:SWext}] 
Taking $h(t) = 1$ in Proposition~\ref{p:epconvonboundeddomains}, we know that, for any $n > 0$,
\[
\frac{{(\log \log x)}^{1/2}}{\log x} \int_{-n}^{n} \biggl| \frac{F_{(x)} (1/2 + it)}{1/2 + it} \biggr|^2 \, dt \xrightarrow{d} \int_{-n}^{n} \frac{g(t)}{|1/2 + it|^2} \lambda(dt) .
\]
Convergence of the full integral follows in two main steps. First, we will prove that the sequence 
\[ 
V_n \coloneqq \int_{-n}^{n} \frac{g(t)}{|1/2 + it|^2} \lambda(dt) ,
\]
converges in distribution, as $n \rightarrow \infty$, and that the limiting distribution satisfies the desired moment bounds. We will then show convergence in distribution of the full integral to this distribution using an approximation lemma. Beginning with the first step, we note that the collection of random variables 
\[
{\Biggl\{ {\biggl( \frac{{(\log \log x)}^{1/2}}{\log x} \int_{-n}^{n} \biggl| \frac{F_{(x)} (1/2 + it)}{1/2 + it} \biggr|^2 \, dt \biggr)}^{q} \Biggr\}}_{x \geq 10, \, n \in \N} ,
\]
is uniformly integrable for any fixed $0 < q < 1$, which follows from Lemma~\ref{l:mcexpectation}. Seeing as we have convergence in distribution for fixed $n \in \N$, it follows from uniform integrability (see, for example,~\cite[Theorem~5.5.9]{Gut}) that
\[
\E {\biggl( \int_{-n}^{n} \frac{g(t)}{|1/2 + it|^2} \lambda(dt) \biggr)}^{q} \leq C,
\]
for some fixed constant $C>0$, depending only on $q$, and in particular it is uniformly in $n$. Subsequently, the sequence of random variables ${(V_n)}_{n \in \N}$ is tight, and by Prokhorov's theorem has a subsequence that converges in distribution, say $V_{n_j} \xrightarrow{d} V_{\mathrm{crit}}$. By monotonicity of $\p (V_n \leq t)$, we have
\[
\p (V_{n_{j+1}} \leq t) \leq \p (V_m \leq t) \leq \p (V_{n_{j}} \leq t) ,
\]
whener $n_j < m \leq n_{j+1}$. Taking $j \rightarrow \infty$, it follows from the sandwich rule that
\[
\p \bigl( V_{\mathrm{crit}} \leq t \bigr) = \lim_{n \rightarrow \infty} \p \bigl( V_n \leq t \bigr),
\]
at every continuity point $t$ of $V_{\mathrm{crit}}$. By our uniform moment bound, it also follows (for example, by~\cite[Theorem~5.5.9]{Gut}) that $\E [V_{\mathrm{crit}}^{q}] \ll_q 1$, for all $0 < q < 1$. \\ 
We now perform the second step, showing convergence in distribution of the full integral to $V_{\mathrm{crit}}$. This will be a consequence of the following elementary lemma:
\begin{lemma}[Approximation result]\label{l:convoffullint}
Suppose that ${(Y_x)}_{x \geq 0}$ is a sequence of real-valued random variables and that we can decompose $Y_x = X_{x,n} + E_{x,n}$ for each $n \in \N$, such that the following assumptions hold:
\begin{enumerate}
\item For each $n \in \N$, $X_{x,n} \xrightarrow{d} Z_n$ as $x \rightarrow \infty$, for some random variables ${(Z_n)}_{n \in \N}$. 
\item $Z_n \xrightarrow{d} Z$ as $n \rightarrow \infty$, for some random variable $Z$. 
\item For any $\varepsilon > 0$, we have $\lim_{n \rightarrow \infty} \limsup_{x \rightarrow \infty} \p ( | E_{x,n} | > \varepsilon ) = 0 $.
\end{enumerate}
Then $Y_x \xrightarrow{d} Z$ as $x \rightarrow \infty$.
\end{lemma}

\begin{proof}[Proof of Lemma~\ref{l:convoffullint}]
This follows from~\cite[Theorem~4.2.8]{Kallenberg}. We have a slightly weaker third condition which follows immediately upon inspecting the proof there.
\end{proof}
In Lemma~\ref{l:convoffullint}, we take $Y_x$ to be the full integral
\[ 
Y_x = \frac{{(\log \log x)}^{1/2}}{\log x} \int_{\R} \biggl| \frac{F_{(x)} (1/2 + it)}{1/2 + it} \biggr|^2 \, dt,
\]
and we write $Y_x = X_{x,n} + E_{x,n}$, where
\[ 
X_{x,n} = \frac{{(\log \log x)}^{1/2}}{\log x} \int_{-n}^n \biggl| \frac{F_{(x)} (1/2 + it)}{1/2 + it} \biggr|^2 \, dt ,
\]
and 
\[
E_{x,n} = \frac{{(\log \log x)}^{1/2}}{\log x} \int_{|t|>n} \biggl| \frac{F_{(x)} (1/2 + it)}{1/2 + it} \biggr|^2 \, dt .
\]
We have seen that $X_{x,n} \xrightarrow{d} V_n$ as $x \rightarrow \infty$ and $V_n \xrightarrow{d} V_{\mathrm{crit}}$ as $n \rightarrow \infty$. This establishes the first two conditions of Lemma~\ref{l:convoffullint}. For the last condition, note (similarly to the end of~\citet[Section~2.4]{HarperLM}) that
\begin{align*}
\E \bigl[ {|E_{x,n}|}^{2/3} \bigr] &\ll {\Biggl( \sum_{|m| > n-1} \frac{1}{m^2} \frac{{(\log \log x)}^{1/2}}{\log x} \int_{m - 1/2}^{m + 1/2} |F_{(x)} (1/2 + it)|^2 \, dt \Biggr)}^{2/3} \\
& \ll \sum_{|m| > n-1} \frac{1}{m^{4/3}} \E {\biggl( \frac{{(\log \log x)}^{1/2}}{\log x} \int_{m - 1/2}^{m + 1/2} |F_{(x)} (1/2 + it)|^2 \, dt \biggr)}^{2/3} \\
& = \sum_{|m| > n-1} \frac{1}{m^{4/3}} \E {\biggl( \frac{{(\log \log x)}^{1/2}}{\log x} \int_{-1/2}^{1/2} |F_{(x)} (1/2 + it)|^2 \, dt \biggr)}^{2/3},
\end{align*}
where the last line follows by translation invariance in law. As noted in the proof of Lemma~\ref{l:mcexpectation}, it follows from the main result proved in the section `Proof of the upper bound in Theorem 1, assuming Key Propositions 1 and 2' of~\cite[Section~4.1]{HarperLM} that the above expectation is bounded uniformly in $x$. Hence, performing the sum over $m$, we find that
\[
\E \bigl[ {|E_{x,n}|}^{2/3} \bigr] \ll n^{-1/3},
\]
uniformly in $x$. Applying Markov's inequality gives $\p (|E_{x,n}| > \varepsilon) \ll \varepsilon^{-2/3} n^{-1/3}$ uniformly in $x$, and the third condition of Lemma~\ref{l:convoffullint} follows immediately. Since all condition of the lemma are satisfied, we have $Y_x \xrightarrow{d} V_{\mathrm{crit}}$ as $x \rightarrow \infty$. This completes the proof of Theorem~\ref{t:SWext}.
\end{proof}

\section{Proof of Theorem~\ref{t:main} and corollaries}\label{s:completion}
\subsection{Proof of Theorem~\ref{t:main}} To complete the proof of Theorem~\ref{t:main}, we combine Proposition~\ref{p:steincna} and Theorems~\ref{t:variance} and~\ref{t:SWext}.
\begin{proof}[Proof of Theorem~\ref{t:main}]
It follows from Proposition~\ref{p:steincna} that, for any complex region $E = \{ x + iy \colon \Re (x) < a, \Im (y) < b \}$, we have
\begin{equation}\label{equ:gaussianbehav}
\p \Biggl( \frac{{(\log \log x)}^{1/4}}{\sqrt{x}} \sum_{\substack{n \leq x \\ P(n) > \sqrt{x}}} f(n) \in E \Biggr) = \p \bigl( Z \in E \bigr) + o(1) ,
\end{equation}
where $Z \sim \sqrt{\mathcal{V}_1 (x)} \, \mathcal{CN} (0,1)$, where $\mathcal{V}_1 (x)$ is independent of $\mathcal{CN} (0,1)$, and
\[ 
\mathcal{V}_1 (x) = \frac{{(\log \log x)}^{1/2}}{x} \sum_{\sqrt{x} < p \leq x} \Bigl| \sum_{m \leq x/p} f(m) \Bigr|^2 .
\]
Recall that
\[ 
\mathcal{V}_5 (x) = \frac{e^{-\gamma} \log 2}{2\pi} \frac{{(\log \log B(x))}^{1/2}}{\log B(x)} \int_{\R} \biggl| \frac{F_{y} (1/2 + it)}{1/2 + it} \biggr| \, dt ,
\]
where $y = \lfloor 100 \log \log \log x \rfloor$ and $B(x) = \sqrt{x}^{e^{-y}}$. We have
\[ 
\mathcal{V}_1 (x) = \mathcal{V}_5 (x) + \bigl( \mathcal{V}_1 (x) - \mathcal{V}_5 (x) \bigr),
\]
and by Theorem~\ref{t:variance}, the term in the parenthesis converges to zero in probability. Therefore, to understand the distribution of $\mathcal{V}_1 (x)$, it suffices to understand the limiting distribution of $\mathcal{V}_5 (x)$, which is given by Theorem~\ref{t:SWext}. Subsequently, we have
\[ 
\mathcal{V}_1 (x) \xrightarrow{d} C V_{\mathrm{crit}} ,
\]
where $V_{\mathrm{crit}}$ is as defined in Theorem~\ref{t:SWext}, and $C = \frac{e^{- \gamma} \log 2}{2 \pi}$. By the continuous mapping theorem~\cite[Theorem~5.10.4]{Gut}, convergence in distribution is preserved by continuous functions, so we have
\[ 
\sqrt{\mathcal{V}_1 (x)} \xrightarrow{d} \sqrt{C V_{\mathrm{crit}}} .
\]
It then follows from independence that $\sqrt{\mathcal{V}_1 (x)} \, \mathcal{CN} (0,1) \xrightarrow{d} \sqrt{C V_{\mathrm{crit}}} \, \mathcal{CN} (0,1)$, as $x \rightarrow \infty$. This can be shown using characteristic functions by conditioning on the complex normal and taking the limit in $x$. The proof is completed by combining this with~\eqref{equ:gaussianbehav}.
\end{proof}

\subsection{Proof of Corollaries~\ref{c:momentconvergence} and~\ref{c:heavytails}}
We first prove Corollary~\ref{c:momentconvergence}, which follows almost immediately from Theorem~\ref{t:main} and uniform integrability.
\begin{proof}[Proof of Corollary~\ref{c:momentconvergence}]
Fix any $0 < q < 2$. It follows from Lemma~\ref{l:mcexpectation} that the sequence
\[
\Biggl\{ {\biggl(\frac{{(\log \log x)}^{1/4}}{\sqrt{x}} \Bigl| \sum_{\substack{n \leq x \\ P(n) > \sqrt{x}}} f(n) \Bigr| \biggr)}^{q} \Biggr\}_{x \geq 2}
\]
is uniformly integrable. By Theorem~\ref{t:main} and the continuous mapping theorem~\cite[Theorem~5.10.4]{Gut}, we have 
\[
\frac{{(\log \log x)}^{1/4}}{\sqrt{x}} \Bigl| \sum_{\substack{n \leq x \\ P(n) > \sqrt{x}}} f(n) \Bigr| \xrightarrow{d} \sqrt{C V_{\mathrm{crit}}} |Z| ,
\]
as $x \rightarrow \infty$, where $C$ and $V_{\mathrm{crit}}$ are as in Theorem~\ref{t:main} and are independent of $Z \sim \mathcal{CN} (0,1)$. Applying~\cite[Theorem~5.5.9]{Gut}, it follows that
\begin{align*}
\lim_{x \rightarrow \infty} \E {\biggl(\frac{{(\log \log x)}^{1/4}}{\sqrt{x}} \Bigl| \sum_{\substack{n \leq x \\ P(n) > \sqrt{x}}} f(n) \Bigr| \biggr)}^{q} &= \E {\Bigl[ \sqrt{C V_{\mathrm{crit}}} |Z| \Bigr]}^{q} , \\
& = C^{\frac{q}{2}} \E \bigl[ V_{\mathrm{crit}}^{\frac{q}{2}} \bigr] \E {|Z|}^{q},
\end{align*}
where the last line follows from independence. It is a standard result that $\E |Z|^q = \Gamma (\frac{q}{2} + 1)$, where $\Gamma$ is the gamma function. This completes the proof of Corollary~\ref{c:momentconvergence}.
\end{proof}
We now prove Corollary~\ref{c:heavytails}. To prove the upper bound, we will make use of the following tail estimate for products of independent random variables (also found in~\cite[Lemma~2.12]{WongTail}).
\begin{lemma}\label{l:prodtailest}
Let $U$ and $V$ be nonnegative independent random variables. Suppose that there exists $C>0$ and $q > 0$ such that
\[ 
\p (U > t) \leq \frac{C}{t^q} ,
\]
for all $t > 0$, and 
\[ 
\E [ V^p ] < \infty
\]
for some $p > q$. Then 
\[ 
\p ( U V > t ) \leq \frac{C \E [ V^q ]}{t^q} ,
\]
for all $t > 0$.
\end{lemma}
\begin{proof}
Let $\E_U$ denote the expectation over $U$ and $\E_V$ the expectation over $V$. We have
\[
\p ( U V > t ) = \E_V \E_U [ \mathbf{1}_{UV>t} ] = \E_V \p \bigl( U > t/V \bigr) \leq \E_V \biggl[ \frac{C V^q}{t^q} \biggr] = \frac{C \E [ V^q ]}{t^q} ,
\]
where the inequality follows from the first assumption of the lemma.
\end{proof}
\begin{proof}[Proof of the upper bound in Corollary~\ref{c:heavytails}]
Fix any small $0 < \varepsilon < 1$. By Theorem~\ref{t:main}, we need to show that $ \p \bigl( | X Y | > y \bigr) \ll_{\varepsilon} 1/y^{2-\varepsilon}$, where
\[ 
X \sim \mathcal{C N} (0,1), \quad
Y \sim \sqrt{C V_{\mathrm{crit}}},
\]
where $C = \frac{e^{-\gamma} \log 2}{2 \pi} $. From Theorem~\ref{t:SWext} we know that $\E \bigl[ Y^{2-\varepsilon} \bigr] \ll_{\varepsilon} 1$, and it follows from Markov's inequality that 
\[ 
\p (Y > y) \ll_\varepsilon \frac{1}{y^{2 - \varepsilon}}.
\]
Since $X$ and $Y$ are independent, it follows from Lemma~\ref{l:prodtailest} that
\[ 
\p \bigl( |X Y| > y \bigr) \ll_\varepsilon \frac{\E \bigl[ |X|^{2 - \varepsilon} \bigr]}{y^{2 - \varepsilon}}.
\]
Seeing as $\E \bigl[ |X|^{2 - \varepsilon} \bigr]$ is uniformly bounded for all $0 < \varepsilon < 1$, we obtain the desired bound.
\end{proof}

\begin{proof}[Proof of the lower bound in Corollary~\ref{c:heavytails}]
To prove the lower bound result, we will make use of the following tail estimate for critical chaos measures: 
\begin{theorem}[Wong~\cite{WongTail}]\label{t:wonggmctail}
For any open set $A \subseteq [-2,2]$, say, such that $\mathrm{Leb} (\partial A) = 0$, and any nonnegative continuous function $f$ on $A$, we have
\[
\p \biggl( \int_{A} f(t) \lambda(dt) > y \biggr) = \frac{1}{y \sqrt{\pi}} \int_{A} f(t) \, dt + o(y^{-1}).
\]
\end{theorem}
This theorem holds more generally for $A$ inside bounded domains. Note that it applies to the specific critical Gaussian multiplicative chaos measure $\lambda (dt)$ appearing in our work. Alternatively, one could use a similar result to Lemma~\ref{l:prodtailest} for the lower bound (see~\cite[Lemma~2.12]{WongTail} for such a result). \\
Fix any $\varepsilon > 0$. We wish to show that $\p \bigl( | X Y | > y \bigr) \gg_\varepsilon 1/y^{2+\varepsilon}$, where
\[ 
X \sim \mathcal{C N} (0,1), \quad
Y \sim \sqrt{C V_{\mathrm{crit}}},
\]
are independent random variables given in Theorem~\ref{t:main}, and $C = \frac{e^{-\gamma} \log 2}{2 \pi} $. By independence and monotonicity of $\p (V_n \geq t)$ (which are increasing in $n$ with limit $\p (V_{\mathrm{crit}} \geq t)$) in Theorem~\ref{t:SWext}, we have
\[
\p \bigl( | X Y | > y \bigr) \geq \p \bigl( \sqrt{C}|X| > 1 \bigr) \p \bigl( V_{\mathrm{crit}} > y^2 \bigr) \gg \p \Biggl( \int_{0}^{1} \frac{g(t)}{|1/2 + it|^2} \lambda(dt) > y^2 \Biggr).
\]
So it suffices to lower bound the probability 
\[ 
\p \Biggl( \int_{0}^{1} g(t) \lambda(dt) > 4 y^2 \Biggr).
\]
We recall the fact that $\| 1 / g \|_{L^{\infty} [0,1]} $ possesses moments of all orders, so we lower bound the above by
\[
\p \Biggl( \frac{1}{\| 1 / g \|_{L^{\infty} [0,1]}} \int_{0}^{1} \lambda(dt) > 4 y^2 \Biggr) \geq \p \Biggl( \frac{1}{\| 1 / g \|_{L^{\infty} [0,1]}} > 4 y^{-\varepsilon} \text{ and } \int_{0}^{1} \lambda(dt) > y^{2+\varepsilon} \Biggr) .
\]
Seeing as $\p (A \cap B) \geq \p (A) - \p (B^c)$, we find that the right-hand side is bounded below by
\[ 
\p \Biggl( \int_{0}^{1} \lambda(dt) > y^{2+\varepsilon} \Biggr) - \p \Biggl( \| 1 / g \|_{L^{\infty} [0,1]} \geq y^{\varepsilon} / 4 \Biggr).
\]
The first probability is $\gg 1/y^{2+\varepsilon}$ by Theorem~\ref{t:wonggmctail}, while an application of Markov's inequality allows us to find that the second term is $\ll_{\varepsilon} 1/y^{100}$, say. This completes the proof of the lower bound.
\end{proof}

\subsubsection*{Acknowledgments}
The author would like to thank his supervisor, Adam Harper, for invaluable discussions throughout this project and for comments on a previous version of this paper. The author would also like to thank Eero Saksman for useful discussions regarding Section~\ref{s:distribofvar}, and is grateful to Ofir Gorodetsky, Christian Webb and Mo Dick Wong for valuable comments.

\subsubsection*{Rights Retention}
For the purpose of open access, the author has applied a Creative Commons Attribution (CC-BY) licence to any Author Accepted Manuscript version arising from this submission.

\printbibliography
\end{document}